\documentclass[12pt]{amsart}

\usepackage[utf8]{inputenc}
\usepackage{url}
\usepackage{hyperref}

\usepackage{graphicx}
\usepackage{amssymb}
\usepackage{float}
\usepackage{fullpage}
\usepackage{comment}
\let\olddiv\div
\usepackage{mathabx}
\let\div\olddiv
\usepackage{mathtools}
\usepackage{color}
\newtheorem{theorem}{Theorem}[section]
\newtheorem{lemma}[theorem]{Lemma}

\newtheorem{proposition}[theorem]{Proposition}
\newtheorem{conjecture}[theorem]{Conjecture}

\theoremstyle{definition}
\newtheorem{definition}[theorem]{Definition}

\theoremstyle{remark}

\numberwithin{equation}{section}

\newcommand{\R}{{\mathbb R}}

\newcommand{\Z}{{\mathbb Z}}
\newcommand{\N}{{\mathbb N}}
\newcommand{\1}{{\mathbf 1}}

\newcommand{\eps}{\varepsilon}

\newcommand{\E}{{\mathcal E}}

\renewcommand{\Re}{\mathop{\mathrm Re}}
\renewcommand{\Im}{\mathop{\mathrm Im}}
\renewcommand{\exp}{\mathrm{exp}}

\title{The exceptional set of the Goldbach problem }
\author[Gautami Bhowmik]{Gautami Bhowmik}
\address{Laboratoire Paul Painlev\'e LABEX-C2EMPI , Universit\'e de Lille, Batiment M2, 59655 Villeneuve-d'Ascq Cedex, France}
\email{gautami.bhowmik@univ-lille.fr}
\author[Lasse Grimmelt]{Lasse Grimmelt}
\address{Department of Pure Mathematics and Mathematical Statistics, University of Cambridge, Cambridge CB3 0WB, UK}
\email{lpg31@cam.ac.uk }

\dedicatory{Dedicated to J\`anos Pintz on his 75th birthday.}
\begin{document}

\begin{abstract} We study the estimates for the number of exceptions to the representation of integers as the sum of at most two prime numbers. Most of this article is a survey that gives an overview of existing results. We begin with the legendary Hardy-Littlewood circle method and show how it paved the way to a power saving by Montgomery-Vaughan in 1975 and Pintz in 2018. We conclude with a new result that is a fully explicit formula for the major arcs. Another new observation is
the non-existence of exceptional zeros under a sparse
 version of the Hardy-Littlewood conjecture. The survey part of this article aims to be accessible to an audience that has not encountered these techniques before. 
\end{abstract}
\maketitle

\section{Introduction}

The Goldbach problem dates back to 1742 and asks if every even integer greater than $2$ can be expressed 
as the sum of two prime numbers. 
This seemingly innocuous looking question remains unanswered almost three centuries later and joins other famous
``simple" unproved statements in additive number theory like the twin prime problem of ascertaining if there exist infinitely many primes $p$ such that $p+2$ is also prime. 
Nevertheless, impressive progress has been made on the binary Goldbach problem, part of which we will give an exposition in this article.  

The Goldbach conjecture 
is  empirically supported by calculations for all numbers  
up to \(4\cdot 10^{18}\) \cite{OHP2014} at which point it seems to hit a computational bottleneck.  
The conjecture is also statistically supported by showing that the set where the conjecture may fail is of density zero,
 in other words ``small". These possible exceptions to the conjecture 
are given by the exceptional set
\begin{align}
    \E(X)=\{2<N\le X, N\in2\N; \ N\neq p_{1}+p_{2}\}
\end{align}
where $p_1$ and $p_2$ are elements of $\mathbb P$ , the set of prime numbers.
The size of the exceptional set thus depends on $X$, which is allowed to become ``large",  and
Goldbach's conjecture is equivalent to the assertion that \(|\E(X)|=0\) for any \(X\). This paper deals with upper bounds of $\E(X)$ and shows how new tools and concepts in analytic number theory
contributed to the evolution of these bounds. To be more precise, we are interested in bounds of the shape $|\E(X)|\leq G(X)$ for some function $G(X)$ that grows slower than $X$. If for example $G(X)=X^{3/4}$, we would get that of the $X/2$ even integers up to $X$, $X/2-X^{3/4}$ can be written as the sum of two primes. Since $\frac{X^{3/4}}{X/2}$ goes to $0$ as $X$ goes to infinity, the number of exceptions is sparse and we say that \emph{almost all} even integers are the sum of two primes. 

One reason why it is interesting to study the exceptional set is that it connects directly to ternary additive problems involving two primes. Roughly speaking, if we take two primes and a third summand from any set with more than $|\E(X)|$ elements, we can hope to represent all integers (up to congruence conditions) in this way. Thus, for example, the standard proof that $|\E(X)|\leq X/(\log X)^2$ can be modified slightly to show that large odd integers are the sum of three primes (since primes are only $1/\log X$ sparse, as we will see in more detail later). This gives the central motivation for this article: we showcase the historic developments that have led to improved estimates for $|\E(X)|$, in particular the substantial contribution of János Pintz. 

The tools that have had success in this problem fall into the area of analytic number theory, and are usually more suited to \emph{counting} the number of solutions, rather than showing simple existence. For this reason, a central object is the Goldbach function
\begin{equation*}
r_2'(N)=\sum_{\substack{p_1+p_2=N \\ p_i \in \mathbb P}} 1
\end{equation*}
that counts the number of representations of $N$ as the sum of two primes. If we can show that it 
is positive for all even $N$, the Goldbach conjecture follows. This is clearly out of reach; but if we can show that it is strictly positive for all integers outside a set of size $G(X)$, we obtain an exceptional set bound as above. Two famous tools are central for this counting: The circle-method and study of the Riemann-zeta function.

\subsection{Outline}
About a hundred years ago Hardy and Littlewood \cite{HL23} conjectured that $r_2'(N)$ is asymptotically equivalent to the product of $\frac{N}{(\log N)^2}$ and a so-called singular series. In section 2 we state this conjecture and connect it to the Riemann zeta function and related $L$-functions. 

Since these functions are not yet understood completely, we may need hypotheses on their zeros to estimate the exceptional set.
Perhaps  the
most widely known of such conjectures is the  Generalized Riemann Hypothesis (GRH). 
If $\chi$ is a Dirichlet character modulo $q$  and $L(s, \chi)$ 
the associated Dirichlet $L$-function, the real parts of its  non-trivial 
zeros are expected, according to the GRH, to lie on the line $\Re (s)=\frac{1}{2} $.
Under this hypothesis, the upper bound  of $|\E(X)|\leq X^{1/2+\eps}$ for any $\eps>0$ was obtained for the exceptional set by Hardy and Littlewood. In Section 4 we explain how the circle method helped Hardy-Littlewood obtain such a bound. Our proofs are not detailed but trace the main ideas. A reader interested in more details about the circle method could consult \cite{Vau1997}. Notice that it was only in 1992 that Goldston \cite{Goldston1992} could replace the $\eps$ above by a logarithm. 

Just after 1937 when  Vinogradov's estimates  became available, Van der Corput \cite{vdC1937} Chudakov \cite{C1937} and Estermann  \cite{E1938}
 independently obtained the first unconditional estimate showing that  $\E(X)$ grows strictly slower than $X$. We explain Vinogradov's method and a key result of Siegel required for its application in Section 4.

In 1975 Montgomery and Vaughan \cite{MV1975} reduced the power of $X$ by using an effective form of Gallagher's work on the distribution of zeros of $L$-functions and showed that there exists a positive effectively computable constant $\delta>0$ such that $\E(X)$ is bounded by $  X^{1-\delta} $ for all large $X$. This is treated in Section 5.

Pintz's work \cite{P2018},\cite{P2023} continues on the lines of Montgomery-Vaughan and obtains an explicit $\delta=0.28$. Section 6 
explains this development. Pintz's key ingredient is a refinement of the contribution of zeros of $L$-functions (see section 4). We will see his approach in Section 6 of this paper.

On our part, we obtain a formula that is completely explicit and generalizes Pintz's result in Section 7. The main idea is to use a smooth form for the major arcs.

The existence of any non-trivial zero outside the conjectured line  would  serve as  a counterexample to
the GRH. One such eventual bad zero, now called the Siegel zero,
would be a real one associated to a unique primitive 
quadratic Dirichlet character and its eventual presence or likely absence plays a significant role in the study of the exceptional set. In Section 8, we obtain a new result that relates Goldbach representations to the non-existence of Siegel-zeros.

\subsection{Notation}
The following will often be used for comparing growths of positive valued functions $f$ and $g$ as $x$ goes to infinity and with $C$ being a constant.
\begin{align*}
  f(x)  \sim g(x)\ &\text {if}\quad \lim_{x\rightarrow \infty}\frac{f(x)}{g(x)}=1,\\
  f(x)=o(g(x))\ &\text{if}\quad \forall C>0, \exists x_0 : |f(x)|\le Cg(x)\  \forall\ x\ge x_0,\\
  f(x)=O(g(x))\ &\text{if}\quad \exists C>0: |f(x)|\le Cg(x)\ \forall\ x.\\
\end{align*}
The symbol $\ll$ may be used instead of the big-oh above and all occurrences of constants  are to be understood with an existence quantor before the statement, in particular they may differ in each occurrence.

\section{Heuristics}

\subsection{Heuristics for binary Goldbach}

To estimate the binary Goldbach representation function, one first needs a model for how often an integer \(n\leq X\) is prime. The most basic heuristic, suggested by the Prime Number Theorem,
\begin{align*}
    \pi(X):=\sum_{p\leq X}1 \sim \frac{X}{\log X},
\end{align*}
is that for integer below \(X\), being prime is an independent random event with probability about \(1/\log X\). A slightly more refined form is
\begin{align*}
    \pi(X)=\sum_{n\leq X}\frac{1}{\log n}\bigl(1+o(1)\bigr),
\end{align*}
which may be interpreted heuristically as saying that
\[
\1_{n\text{ is prime}} \approx \frac{1}{\log n}.
\]
In other words, $n$ being prime is replaced in the problem of counting primes up to $X$ by a independent random event of probability  $\frac{1}{\log n}$. Thus, we can think of the simple function $\frac{1}{\log n}$ as our first model for the primes. At a first glance, this seems like an absurd oversimplification when our goal is to understand Goldbach type problems, and indeed we will see that it is, but nevertheless it is a good first step. 

This inverse logarithmic density makes it convenient to replace the prime indicator by
\[
\1_{n\text{ is prime}} \log n,
\]
which (up to a contribution of higher prime powers that we ignore in this survey) is the von-Mangoldt function  $\Lambda(n)$. The logarithmic weight compensates the sparsity precisely and we end up with the simplest model expectation
\[
\Lambda(n)\approx 1.
\]

We include the same weight in the count for Goldbach representations. Thus we modify $r'_2(N)$ and write
\begin{align*}
    r_2(N):=\sum_{n_1+n_2=N}\Lambda(n_1)\Lambda(n_2).
\end{align*}
The model \(\Lambda(n)\approx 1\) leads to the first naive approximation
\begin{align}\label{eq:heuristicI}
    r_2(N)\approx \sum_{n_1+n_2=N}1 \sim N.
\end{align}
Equivalently, one may expect
\[
r_2'(N)\approx \sum_{n_1+n_2=N}\frac{1}{(\log n_1)(\log n_2)}
\sim \frac{N}{(\log N)^2}.
\]

However, \eqref{eq:heuristicI} cannot be correct as stated. Apart from \(2\), all primes are odd, so an odd integer \(N\) can be written as a sum of two primes only if one of the primes is \(2\). Thus for odd \(N\) one has
\begin{align}\label{eq:ODD}
    r_2'(N)=
    \begin{cases}
        2,&\text{if }N=p+2\text{ for some prime }p,\\
        0,&\text{otherwise.}
    \end{cases}
\end{align}
This is plainly incompatible with the heuristic above. The point is that the model
\[
   \Lambda(n)\approx 1
\]
works reasonably well for counting primes on average, but is too crude for additive questions such as Goldbach's problem.

We therefore try to incorporate local congruence information. The first correction is parity: primes greater than \(2\) are odd, so for the weighted indicator \(\Lambda(n)\) it is natural to replace the crude model \(1\) by
\begin{align*}
    \Lambda_{\mathrm{heur}}(n)\approx
    \begin{cases}
        2,&\text{if }n\text{ is odd},\\[1ex]
        0,&\text{if }n\text{ is even},
    \end{cases}
\end{align*}
which preserves the correct average order. This can be written more compactly as
\begin{align}\label{eq:q=2mod}
    1+(-1)^{n+1}.
\end{align}

There is no reason to stop at the modulus \(2\). For any small prime \(q\leq R\), all other primes are not divisible by \(q\), so a more realistic model should vanish on multiples of \(q\). There being no apparent reason for any bias, we may expect that the remaining mass is distributed equally among the \(q-1\) reduced residue classes modulo \(q\). This suggests a correction of the form
\begin{align}\label{eq:modelqprime}
    1+\sum_{\substack{q\leq R\\ q \text{ prime}}} f_q(n),
\end{align}
where, generalising $(-1)^{n+1}$ in \eqref{eq:q=2mod},
\begin{align*}
    f_q(n)=
    \begin{cases}
        -1,&\text{if } q\mid n,\\[1ex]
        \dfrac{1}{q-1},&\text{if } q\nmid n.
    \end{cases}
\end{align*}
Note that \(f_q\) is periodic modulo \(q\) and has mean value \(0\):
\begin{align*}
    \sum_{a=1}^{q} f_q(a)=0.
\end{align*}

There is, however, an immediate problem with \eqref{eq:modelqprime}. If \(n\) is divisible by both \(2\) and \(3\), then \(f_2(n)=f_3(n)=-1\), so the model predicts a negative value. This shows that the local corrections should not be added independently, we need to compensate for the fact that we have twice taken out multiples of two primes. More generally, we need to continue by inclusion--exclusion based on number of small prime factors, which is nothing else than continuing this process from primes to products of two primes, products of three primes and so on. The idea of inclusion--exclusion is crucial to another fundamental tool of modern analytic number theory: Sieves, which for example played a major role in the breakthroughs on bounded gaps between primes \cite{zhang}, \cite{maynard}. 

The right language to formalise inclusion--exclusion in our context is provided by Ramanujan sums given by
\begin{align*}
    c_q(n):=\sum_{\substack{a \bmod q\\ (a,q)=1}} e\!\left(\frac{an}{q}\right),
    \qquad e(x):=e^{2\pi i x}.
\end{align*}
We will encounter these sums again and again throughout our Goldbach journey. When \(q\) is prime, one has
\[
c_q(n)=
\begin{cases}
q-1,&\text{if }q\mid n,\\
-1,&\text{if }q\nmid n,
\end{cases}
\]
and hence
\begin{align}
    \label{eq:ramsum}-\frac{c_q(n)}{q-1}=
\begin{cases}
-1,&\text{if }q\mid n,\\[1ex]
\dfrac{1}{q-1},&\text{if }q\nmid n,
\end{cases}
=f_q(n).
\end{align}

Thus the functions \(f_q\) are for primes $q$ precisely the basic Ramanujan-sum corrections, after dividing by $q-1$ which corresponds to the number of residue classes not divisible by $q$. To complete the model we need to generalise this count of non-divisible residue classes, as well as a function that alternates the sign depending on the number of prime factors. These are standard and given respectively by Euler's totient function
\begin{align*}
    \varphi(q)=q\prod_{p|q}(1-1/p).
\end{align*}
and the Möbius function $\mu(q)$, where \(\mu(q)=0\) if \(p^2\mid q\) for some prime \(p\), and otherwise
\begin{align*}
    \mu(q)=(-1)^{\#\{p:p\mid q\}}.
\end{align*}
We then can define our additive model as (see \cite{heath-brown} for one of the early occurences of this function)
\begin{align}\label{eq:FRmodel}
    \Lambda_R(n):=\sum_{q\leq R}\frac{\mu(q)c_q(n)}{\varphi(q)}.
\end{align}

We remark that since both \(\mu(q)\) and \(c_q(n)\) are multiplicative in \(q\), the sum in \eqref{eq:FRmodel} is the truncated version of the Euler product
\begin{align}\label{eq:cramer}
    \prod_{p\leq R}\left(1-\frac{c_p(n)}{p-1}\right),
\end{align}
which is, by \eqref{eq:ramsum}, a weighted indicator of \(n\) having no small prime divisor. One can, however, not replace the sum by the product in general.

From this model, it is clear that the expected size of \(r_2(N)\) should depend on the residue class of \(N\) modulo small primes. Indeed, a simple calculation with Ramanujan sums shows that
\begin{align}\label{eq:truncated-singular-series}
    \sum_{n_1+n_2=N} \Lambda_R(n_1)\Lambda_R(n_2)
    =N\,\mathfrak{S}_R(N)\,(1+o(1)),
\end{align}
where
\begin{align}\label{eq:truncated-SS}
     \mathfrak{S}_R(N)=\sum_{q\le R}
     \frac{\mu(q)^2 c_q(N)}{\varphi(q)^2}.
\end{align}
Here, in contrast to \eqref{eq:cramer}, we can add terms $q>R$ with neglibile error. The infinite sum can then be rigorously written in product form, and we obtain the singular series (we will motivate this terminology in the next section)
\begin{align}\label{eq:SS}
    \mathfrak{S}(N)=\prod_p \left(1+\frac{c_p(N)}{(p-1)^2}\right)=\prod_{p\nmid N}\left(1-\frac{1}{(p-1)^2}\right)
    \prod_{p\mid N}\left(1+\frac{1}{p-1}\right).
\end{align}

For odd \(N\), the factor at \(p=2\) vanishes, reflecting the fact that there should be no main term in that case. For even \(N\), the factor at \(p=2\) equals \(2\), and \eqref{eq:SS} becomes
\[
\mathfrak{S}(N)
=
2\prod_{\substack{3\leq p\\ p\nmid N}}
\left(1-\frac{1}{(p-1)^2}\right)
\prod_{\substack{3\leq p\\ p\mid N}}
\left(1+\frac{1}{p-1}\right),
\]
which is bounded away from $0$, thus giving us many expected representations. Hardy--Littlewood used the singular series to conjecture an asymptotic for the number of Goldbach representations. 

\begin{conjecture}[Hardy--Littlewood] \label{conj:HL}
For even integers \(N\to\infty\), one has
\begin{align*}
    r_2(N)\sim \mathfrak{S}(N)N.
\end{align*}
Equivalently,
\begin{align*}
    r_2'(N)\sim \mathfrak{S}(N)\frac{N}{(\log N)^2}.
\end{align*}
\end{conjecture}
In particular the truth of this conjecture would imply that as $X\to \infty$, the exceptional set $\mathcal{E}(X)$ stays finite. 

\subsection{Counting primes in arithmetic progressions}

As we have seen, the model \(\Lambda_R(n)\) rests on two assumptions about the distribution of primes:
\begin{enumerate}
    \item the density of primes around \(n\) is \(1/\log n\);
    \item primes are equally distributed in admissible residue classes.
\end{enumerate}
Both assertions can be approached with the theory of $L$-functions. The big breakthrough came from Riemann, who in his 1859 memoir for the first time considered the function
\begin{align*}
    \zeta(s)=\sum_{n=1}^\infty n^{-s}\qquad (\Re(s)>1),
\end{align*}
now named after him, for a complex variable $s$ and used analytic continuation to extend it beyond the range $\Re(s)>1$. Its  zeros in the complex plane govern the distribution of primes, and one can prove an explicit formula, convergent in a restrictive sense, of the shape
\begin{align*}\label{eq:Lambdaexplicit}
    \sum_{n\leq X}\Lambda(n)
    =
    X
    -
    \sum_{\rho}\int_2^X t^{\rho-1}\,dt
    + O(\text{lower-order terms}),
\end{align*}
where \(\rho\) runs over the non-trivial  zeros of \(\zeta(s)\) in the critical strip \(0\leq \Re(\rho)\leq 1\). 
Thus, recalling our logarithmic normalisation, the main term $X$ comes from the pole of $\zeta$ at $s=1$ and corresponds to the expected local density of primes \(1/\log n\), 
while the  zeros contribute oscillating correction terms. We can think about this in terms of a refined model 
\begin{align*}
    \Lambda(n)\approx 1 - \sum_{\rho} n^{\rho-1}.
\end{align*}
For simplicity, let us only look at the effect of a single zero \(\rho=\beta+i\gamma\), where \(\beta,\gamma\in\mathbb R\).  One has
\begin{align*}
    n^{\rho-1}=n^{\beta-1+i\gamma}=n^{\beta-1} e^{i\gamma\log n}.
\end{align*}
Thus, the zero corrupts the model $\Lambda(n)\approx 1$ by a term whose size depends on $\Re(\rho)$ (since $|n^{\beta-1} e^{i\gamma\log n}|=n^{\beta-1}$ ) and that oscillates with a frequency depending on the imaginary part\footnote{Non-real oscillations cancel in conjugate pairs, since if \(\beta+i\gamma\) is a zero, then so is \(\beta-i\gamma\).}. Consequently, if $\beta$ is close to $1$, such a zero would increase the number of primes in certain ranges, while decreasing it in others. 

This explains the importance of the real part of the zeros: if it is small, then the additional oscillating terms have small absolute value. It follows that the assumed density (1) is closely related to the fact that \(\zeta(s)\) has no zeros too close to the line \(\Re(s)=1\), so that no additional term is as large as the main term. The Prime Number Theorem itself is equivalent to the absence of zeros on the line \(\Re(s)=1\), and this was first proved independently by Hadamard and de la Vall\'ee Poussin.

One can show that if $\rho$ is a zero, so is $1-\rho$, thus the best we can hope for is that $\beta=1/2$, which is precisely what the famous Riemann Hypothesis asserts. Together with standard information on the number of zeros, it implies that
\begin{align*}
    \sum_{\rho}\int_2^X t^{\rho-1}\,dt
    = O(X^{1/2+\varepsilon})
\end{align*}
for every \(\varepsilon>0\). We are still very far from proving this.

The underlying reason why \(\zeta(s)\) is connected to primes is that \(n^{-s}\) is multiplicative:
\begin{align*}
    (ab)^{-s}=a^{-s}b^{-s}.
\end{align*}
To study assumption (2), we need multiplicative functions that also encode congruence conditions. A Dirichlet character \(\chi \pmod q\) is a completely multiplicative, \(q\)-periodic function
\[
\chi:\mathbb Z\to \mathbb C
\]
such that \(\chi(n)=0\) when \((n,q)>1\) and \(\chi(1)=1\). The principal character \(\chi_0 \pmod q\) is defined by
\[
\chi_0(n)=\chi_0^{(q)} (n)=
\begin{cases}
1,& (n,q)=1,\\
0,& (n,q)>1.
\end{cases}
\]
The orthogonality relation says that for \(a,b\in\mathbb Z\),
\begin{equation}\label{eq:orth1}
\frac{1}{\varphi(q)}\sum_{\chi \,(\mathrm{mod}\, q)} \chi(a)\overline{\chi(b)}
=
\begin{cases}
1,& a\equiv b \pmod q \text{ and } (a,q)=1,\\
0,& \text{otherwise}.
\end{cases}
\end{equation}
Thus characters allow us to isolate a residue class \(a \pmod q\).

Replacing \(\zeta\) by the Dirichlet \(L\)-function
\begin{align*}
    L(s,\chi)=\sum_{n=1}^\infty \chi(n)n^{-s}\qquad (\Re(s)>1),
\end{align*}
with the help of \eqref{eq:orth1}, one obtains the corresponding explicit formula for primes in arithmetic progressions. 

\begin{align}\label{eq:Lambdaexplicit}
    \sum_{\substack{n\leq X\\ n\equiv a \bmod q}}\Lambda(n)
    =
    \frac{X}{\varphi(q)}
    -
    \frac{1}{\varphi(q)}
    \sum_{\chi \,(\mathrm{mod}\, q)} \overline{\chi(a)}
    \sum_{\rho_\chi}\int_2^X t^{\rho_\chi-1}\,dt
    + O(\text{lower-order terms}),
\end{align}
where \(\rho_\chi\) runs over the non-trivial zeros of \(L(s,\chi)\), with $(a,q)=1$.. 

This has exactly the same interpretation as for \(\zeta(s)\). The main term
\[
\frac{X}{\varphi(q)}
\]
comes from the poles of $L(s,\chi_0)$ at $s=1$ and corresponds to the expected local density of primes in each admissible residue class modulo \(q\), while the  of the various \(L(s,\chi)\) measure the failure of perfect equidistribution. In particular, if one zero \(\rho_\chi=\beta+i\gamma\) lies very close to the line with real part \(1\), then it creates a large oscillating bias in the distribution of primes in residue classes modulo \(q\).

Since the expectation in Conjecture \ref{conj:HL} was based on (1) and (2), a zero \(\rho\) of some \(L(s,\chi)\) with real part close to \(1\) is \emph{bad}, since it changes the expected local density in the model (after normalising by $\varphi(q)$) 
by adding  terms of the shape
\begin{align}\label{eq:correctionterms}
    \overline{\chi(n)}n^{\rho-1}
\end{align}
thereby introducing both periodic fluctuations in the residue class \(\pmod q\) through \(\overline{\chi(n)}\) and oscillations depending on the size of \(n\) through the imaginary part of \(\rho\). In particular, such bad  are incompatible with the model $\Lambda_R$ and with it the Hardy--Littlewood asymptotics. Just as in the case of $\zeta$, we may again expect that all  of $L(s,\chi)$ have real parts $1/2$, or at least have all real parts $<\Theta$ for some value $\Theta$ clearly below $1$. The proof of the existence of any $\Theta<1$ would be a major breakthrough in analytic number theory. In the next section we will assume that $\Theta<3/4$ and show afterwards what can be done to unconditionally deal with potential bad zeros.

\section{Hardy--Littlewood}

In this section we explain how the Hardy--Littlewood method leads to a conditional estimate for the exceptional set in Goldbach's problem. More precisely, we study the weighted binary representation function, introduced above,
\[
r_2(N)=\sum_{n_1+n_2=N}\Lambda(n_1)\Lambda(n_2),
\]
and show that, under a suitable zero-free hypothesis for Dirichlet \(L\)-functions, it is greater than $N \mathfrak{S}(N)/2$ most of the times. Since this gives many representations for $N$ as the sum of two primes, it provides a bound for  \(\E(X)\).

\begin{theorem}[Hardy--Littlewood, conditional exceptional-set bound]\label{thm:HL}
Assume that every zero of every Dirichlet \(L\)-function has real part at most \(\Theta<3/4\). Then for every \(\varepsilon>0\) we have
\begin{align*}
    |\mathcal{E}(X)|\ll_\varepsilon X^{2\Theta-\frac12+\varepsilon}.
\end{align*}
\end{theorem}

\subsection{(Not) going in circles}

The basic idea of the \emph{circle method} is to rewrite the indicator of the condition \(n=0\) in an analytically useful way. A toy approach is the trivial observation that
\begin{align}\label{eq:circlemethod0}
    \1_{\{n=0\}}=\1_{\{n\equiv 0\bmod q\}}\1_{\{|n|<q\}}.
\end{align}
Thus one detects vanishing in the integers by combining a divisibility condition and a size condition. The circle method may be viewed as a refinement of \eqref{eq:circlemethod0}, in which one uses many congruence conditions and many size conditions simultaneously.

Historically the method starts from the residue theorem. For any integer \(k\) and any radius \(R>0\), we can evaluate the complex integral
\begin{align*}
   \frac{1}{2\pi i}\int_{|x|=R}x^{k-1}\,dx
   =
   \begin{cases}
       1,&\text{if }k=0,\\
       0,&\text{otherwise.}
   \end{cases}
\end{align*}
Writing $k=N-n_1-n_2$, this gives a different way of encoding the constraint \(n_1+n_2=N\).

We now give an overview of the proof of Theorem \ref{thm:HL}, following closely the original ideas in \cite{HL23}. Hardy--Littlewood set
\begin{align*}
    f(x):=\sum_{n}\Lambda(n)x^n.
\end{align*}
Then, for \(|x|<1\) to ensure convergence,
\begin{align*}
    f(x)^2=\sum_{N\geq 1} r_2(N)x^N,
\end{align*}
and hence for $0<R<1$ 
\begin{align}\label{eq:r2-circle}
    r_2(N)=\frac{1}{2\pi i}\int_{|x|=R}\frac{f(x)^2}{x^{N+1}}\,dx.
\end{align}

Next, they choose
\[
R=e^{-1/N},
\]
and divide the circle \(|x|=R\) into Farey arcs of order
\[
Q=\lfloor \sqrt{N}\rfloor.
\]
Recall that the Farey sequence of order \(Q\) is the set of all reduced fractions \(a/q\) with the condition \(0\leq a\leq q\leq Q\) and \((a,q)=1\), arranged in increasing order. A key property is that neighbouring fractions are well separated: if
\[
\frac{a'}{q'}<\frac{a}{q}<\frac{a''}{q''}
\]
are consecutive elements, then one has 
\(aq'-a'q=a''q-aq''=1\) and \(q+q''>Q\), and in particular
\[
\frac{1}{q(q+q'')} < \frac{1}{qq'}= \frac{a}{q}-\frac{a'}{q'}.
\]
This spacing property makes Farey sequences a convenient tool to organise rational approximations to a real number, and underlies the decomposition of the circle into arcs centered at fractions \(a/q\) with controlled denominators. Then, if
\[
\frac{a'}{q'}<\frac{a}{q}<\frac{a''}{q''}
\]
are consecutive terms in the Farey sequence of order \(Q\), the Farey arc around \(a/q\) is the interval
\[
\left(\frac{a}{q}-\frac{1}{q(q+q')},\,
      \frac{a}{q}+\frac{1}{q(q+q'')}\right)=\left(\frac{a'+a}{q'+q}, \frac{a+a''}{q+q''}\right).
\]
This interval contains and is contained in the symmetric intervals around
\(a/q\) of radius  \(1/2qQ\) and  \(1/qQ\) respectively.

We now decompose the circle according to which Farey arc the argument of \(x\) belongs to. On the arc around \(a/q\) we write
\begin{align}\label{eq:x-param}
    x=e(a/q)e^{-Y},
\end{align}
where
\begin{align*}
    Y=\eta-2\pi i\beta,\qquad \eta=\frac{1}{N},
\end{align*}
and \(\beta\) ranges over the above mentioned interval of length \(\asymp 1/(qQ)\). A recurrent theme will be the following: the denominator $q$ is related to the $q$ divisibility in the \eqref{eq:circlemethod0} and the $\beta$ integral is related to the size condition.

We note that, since only numerators coprime to the denominator appear in the Farey sequence, the sum over the centre points of all arcs with a fixed denominator \(q\) produces the Ramanujan sum \(c_q(n)\). This is the same object that arose in the heuristic discussion of local congruence obstructions.

For the sketch of the proof, we consider the function $f$ only at the centres of the arcs, that is, at \(\beta=0\). Then we have to evaluate
\begin{align*}
    f\left( e(a/q)e^{-1/N}\right )=\sum_n \Lambda(n) e(an/q)e^{-n/N}.
\end{align*}
Since \(e(an/q)\) is \(q\)-periodic in \(n\), we sort primes into residue classes modulo \(q\). Using the orthogonality of Dirichlet characters, this may done by writing
\begin{align}\label{eq:char-expand}
    e(an/q)
    =
    \frac{1}{\varphi(q)}
    \sum_{\chi \,(\mathrm{mod}\, q)}
    \tau(\bar\chi)\chi(a)\chi(n),
\end{align}
where
\[
\tau(\chi):=\sum_{m \,(\mathrm{mod}\, q)} \chi(m)e(m/q)
\]
is the Gauss sum. Hence
\begin{align*}
    f(e(a/q)e^{-1/N})
    =
    \frac{1}{\varphi(q)}
    \sum_{\chi \,(\mathrm{mod}\, q)}
    \tau(\bar\chi)\chi(a)
    \sum_n \Lambda(n)\chi(n)e^{-n/N}.
\end{align*}

The exponential weight $e^{-n/N}$ localises the sum to primes of size $\asymp N$ and allows one to apply a  version of the explicit formula discussed in the previous section. One can show an analogue of \eqref{eq:Lambdaexplicit} which, since 
\begin{align*}
    \int_0^\infty e^{-t/N} t^{\rho-1}\,dt=\Gamma(\rho)N^{\rho}
\end{align*}
and $\tau(\chi)=\mu(q)$ for the principal character $\chi=\chi_0^{(q)}$, takes the shape
\begin{align}\label{eq:major-arc-centre}
    f(e(a/q)e^{-1/N})
    =
    N\frac{\mu(q)}{\varphi(q)}
    -
    \frac{1}{\varphi(q)}
    \sum_{\chi \,(\mathrm{mod}\, q)}
    \tau(\bar\chi)\chi(a)
    \sum_{\rho_\chi}\Gamma(\rho_\chi)N^{\rho_\chi}
    +O(\text{lower-order terms}),
\end{align}
where \(\rho_\chi\) runs over the non-trivial  zeros of \(L(s,\chi)\). The crucial point is that the main term is exactly what one expects from equidistribution in arithmetic progressions, while the secondary term is governed by the  zeros of the Dirichlet \(L\)-functions.

The \(\Gamma\)-function makes the sum over \(\rho_\chi\) decay rapidly. Thus, under the hypothesis that  \(\Re(\rho_\chi)\leq \Theta\) and using the standard estimate
\begin{align}\label{eq:taubound}
    |\tau(\chi)|\leq q^{1/2},
\end{align}
the contribution of the zeros satisfies
\begin{align}\label{eq:HLerror}
    \left|
    \frac{1}{\varphi(q)}
    \sum_{\chi \,(\mathrm{mod}\, q)}
    \tau(\bar\chi)\chi(a)
    \sum_{\rho_\chi}\Gamma(\rho_\chi)N^{\rho_\chi}
    \right|
    \ll q^{1/2}N^{\Theta+\varepsilon}
    \ll N^{\Theta+\frac14+\varepsilon},
\end{align}
since on the Farey arcs one has \(q\leq Q\asymp N^{1/2}\).

If in \eqref{eq:major-arc-centre} we keep only the main term and substitute it into \eqref{eq:r2-circle}, then the contribution of the arc around \(a/q\) can be shown to be
\begin{align*}
    N\frac{\mu(q)^2}{\varphi(q)^2}e(-aN/q)+O(N^{1/2+\varepsilon}).
\end{align*}
Summing first over reduced residues \(a \bmod q\) produces  Ramanujan sums $c_q(-N)$ and then summing over \(q\) gives 
\begin{align*}
    N\sum_{q\leq Q}\frac{\mu(q)^2}{\varphi(q)^2}c_q(-N)
    =
    N\mathfrak{S}(N)+O(N^{1/2+\varepsilon}),
\end{align*}
where
\[
\mathfrak{S}(N)=\sum_{q=1}^\infty \frac{\mu(q)^2}{\varphi(q)^2}c_q(-N)
\]
is the same singular series that already appeared in the heuristic discussion. This is exactly what one should expect: in both cases the main term comes from the principle that primes are evenly distributed among the reduced residue classes.

We remark that the terminology of singular series can also be understood in this way: We are integrating over a circle with radius getting close to $1$. Our generating function $f(x)$ no longer converges for $|x|=1$. The contribution of singularities on that circle relates to our main term, part of which is the singular series.

Let us recapitulate what has happened so far. We have in \eqref{eq:r2-circle} rewritten $r_2(N)$ as an integral over the circle, involving $f(x)^2$. We have then decomposed the circle into arcs around rational numbers $a/q$. On each of these arcs, we have shown that $f(x)$ has a main term that (at central point of the arcs) is of size 
\begin{align*}
    \frac{N}{\varphi(q)}
\end{align*}
and an error term of size 
\begin{align*}
   q^{1/2}N^{\Theta+\eps}.
\end{align*}
Recall that $q<N^{1/2}$ and in the best case we could hope for $\Theta=1/2$. Then, ignoring the $\eps$, we actually only have for $q<N^{1/3}$ that
\begin{align*}
    \frac{N}{\varphi(q)}>q^{1/2}N^{1/2}.
\end{align*}
Even worse, Parseval implies that the average size of $|f(x)|$ is $N^{1/2}(\log N)^{1/2}$, which is too large for the expected main term. Thus, any resolution of the binary Goldbach problem needs to be supplied with much sharper information. For a more in depth description, we refer to Tao's blog post\footnote{\url{https://terrytao.wordpress.com/2012/05/20/heuristic-limitations-of-the-circle-method/}}.

It is, however, possible to obtain a mean-square estimate. Writing
\[
E_2(N):=r_2(N)-N\mathfrak{S}(N),
\]
and using Parseval's identity together with \eqref{eq:HLerror}, one obtains
\begin{align}\label{eq:meansquare-r2}
    \sum_{N\leq X}|E_2(N)|^2
    \ll X^{\frac32+2\Theta+\varepsilon}.
\end{align}
This is the first point at which the circle method yields genuine information for the binary Goldbach problem. We give more details about mean square bounds in the following sections. 

We can use \eqref{eq:meansquare-r2} and Chebyshev's inequality, or the following simple calculation that uses that $\mathfrak{S}(2N')\gg 1$, to show that for  most even $N$, $E_2(N)$ is small :
\begin{align*}
    \sum_{N'<X/2} \1_{|E_2(2N)|>\eta 2N'\mathfrak{S}(2N')} \leq \sum_{N'\leq X/2} \frac{|E_2(2N')|^2}{\eta^2 (2N')^2 \mathfrak{S}(2N')^2}\ll \eta^{-2} X^{-\frac12+2\Theta+\varepsilon}.
\end{align*}

If we set $\eta=1/2$, this shows that $r_2(N)\geq \frac{N \mathfrak{S}(N)}{2}$ for all even $N\leq X$ with at most $O_\varepsilon(X^{-\frac12+2\Theta+\varepsilon})$ exceptions. In  particular, Theorem \ref{thm:HL} follows and we have shown that as  long as $\Theta<3/4$, almost all even integers are the sum of two primes.

\section{Siegel--Walfisz and Vinogradov}

The Hardy--Littlewood argument of the previous section shows how the circle method leads to a strong conclusion once one has sufficiently good understanding of the distribution of primes in arithmetic progressions. We now explain how two later advances made it possible to replace this conditional argument by an unconditional one. The first is the Siegel--Walfisz theorem, which gives uniform control of primes in arithmetic progressions for small moduli. This allows one to handle those arcs belonging to rational numbers with small denominator. The second is Vinogradov's estimate for exponential sums over primes, which takes care of the remaining arcs.

\subsection{Siegel--Walfisz}

For a fixed Dirichlet character \(\chi\) of modulus \(q\), and for \(s=\sigma+it\), the classical argument of de la Vall\'ee Poussin extends from \(\zeta(s)\) to \(L(s,\chi)\) and yields a zero-free region, with a positive constant $c$,  of the form 

\begin{align*}
    \sigma \geq 1-\frac{c}{\log(q(|t|+2))},
\end{align*}
with at most one exception, which, if it exists, is real, simple, and attached to a real character. Though it does not play a central r\^ole in what follows, we will encounter this possible exceptional character again. We call the associated zero a \emph{Siegel zero}.

Siegel's theorem then shows, ineffectively (the implied constants depend on parameters that cannot be computed), that such a character cannot have small conductor. In particular, for every fixed \(A>0\), once \(X\) is sufficiently large, there is no such character of conductor at most \((\log X)^A\). The relevant consequence for us is the character form of the Siegel--Walfisz theorem: there exists $c>0$ such that for every \(A>0\) and uniformly for every non-principal character \(\chi \pmod q\) with
\begin{align*}
    q\le (\log X)^A,
\end{align*}
one has
\begin{align}\label{eq:SW-char}
    \sum_{n\le X}\Lambda(n)\chi(n)
    \ll_A X e^{-c\sqrt{\log X}},
\end{align}
where the implied constant depends on $A$ and is ineffective.

\subsection{Decomposing the circle again}

We now switch to Vinogradov's formulation of the circle method. Instead of working with an infinite generating function on the circle \(|x|=e^{-1/N}\), we work for $\alpha\in[0,1)$ with exponential sums
\begin{align*}
    S(\alpha):=\sum_{n\le N}\Lambda(n)e(\alpha n).
\end{align*}
This is essentially the same weighted object as in the previous section, now written as a finitely supported Fourier series rather than an infinite power series.

The residue theorem is replaced by the Fourier identity
\begin{align*}
    \int_0^1 e(\alpha n)\,d\alpha=
   \begin{cases}
       1,&\text{if }n=0,\\
       0,&\text{otherwise},
   \end{cases}
\end{align*}
and hence
\begin{align}\label{eq:r2-fourier}
    r_2(N)=\int_0^1 S(\alpha)^2 e(-\alpha N)\,d\alpha.
\end{align}
This is the Fourier-analytic counterpart of \eqref{eq:r2-circle}.

We again decompose \([0,1]\) into arcs around rational points, but now allow the order $Q$ to be more general in the range
\[
\sqrt{N}\leq Q\leq N.
\]
Instead of the variable-width Farey arcs, we split \([0,1]\) into two parts according to a cutoff $R:=N/Q$. For denominators $q\leq R$ we use fixed-width arcs
\begin{align}\label{eq:majorarcdef}
\mathfrak{M}(R)=\bigcup_{\substack{q\leq R\\ (a,q)=1}} \left\{\alpha: \left|\alpha-\frac{a}{q}\right|\leq \frac{R}{qN}\right\},
\end{align}
and call the remainder
\begin{align*}
    \mathfrak{m}(R)=[0,1)\setminus \mathfrak{M}(R).
\end{align*}
We split the integral accordingly:
\begin{align}
 \nonumber r_2(N)=& \int_0^1 S(\alpha)^2 e(-N\alpha)\,d\alpha = \int_{\mathfrak{M}(R)} S(\alpha)^2 e(-N\alpha)\,d\alpha + \int_{\mathfrak{m}(R)}S(\alpha)^2 e(-N\alpha)\,d\alpha\\
    =&: r_{\mathfrak{M}}(N)+r_{\mathfrak{m}}(N).   \label{eq:splitmajorminor}
\end{align}

The arc around $a/q$ (and the related Farey arc previously) has length $R(qN)^{-1}$. The cutoff $q\leq R$ is natural because it distinguishes whether this width is larger or smaller than \(1/N\). The reason  for $1/N$ to be the critical scale can be seen from the simplest model: the geometric sum
\begin{align}\label{eq:geosum}
    \sum_{n\leq N} e(\alpha n)
    =
    e\!\left(\frac{(N+1)\alpha}{2}\right)
    \frac{\sin(\pi N\alpha)}{\sin(\pi\alpha)}
\end{align}
is essentially the sinc kernel.  For \(\|\alpha\|\) smaller than \(1/N\) the sum is of size \(N\) and essentially constant, while for \(\|\alpha\|\) larger than \(1/N\) it begins to decay. In other words, the arc width is inversely proportional to the length of the sum.

On the major arcs, with $R=(\log N)^A$ for a large fixed $A$, one proceeds much as in the previous section: expand near a rational point $a/q$, separate residue classes modulo $q$ using characters, and apply the information on primes in arithmetic progressions to recover the main term. Setting
\[
    \Psi(x,\chi) := \sum_{n\leq x} \Lambda(n)\chi(n),
\]
the Siegel--Walfisz theorem \eqref{eq:SW-char} gives $\Psi(x,\chi) \ll_A x\,e^{-c\sqrt{\log x}}$ for non-principal $\chi \pmod{q}$ with $q\leq R$. By partial summation,
\begin{align}\label{eq:sum-by-parts}
    \sum_{n\leq N}\Lambda(n)\chi(n)e(\beta n)
    = \Psi(N,\chi)\,e(\beta N) - 2\pi i\beta\int_1^N \Psi(t,\chi)\,e(\beta t)\,dt,
\end{align}
and applying the Siegel--Walfisz bound to both terms gives
\[
    \sum_{n\leq N}\Lambda(n)\chi(n)e(\beta n) \ll_A N\,e^{-c'\sqrt{\log N}},
\]
since $|\beta|\leq R/N$ and the exponential decays faster than any power of $\log N$. For the principal character, the prime number theorem gives the main term. Combining across characters via \eqref{eq:char-expand} yields
\begin{align}\label{eq:SW-major-local}
    S(\alpha)
    =
    \frac{\mu(q)}{\varphi(q)}\sum_{n\leq N}e(\beta n)
    +O\!\left(N\exp(-c'\sqrt{\log N})\right),
\end{align}
uniformly for $q\leq R$ and $\beta = \alpha - a/q$ in the corresponding arc. Here the main term reflects exactly the heuristic from Section~2: after normalisation by $\varphi(q)$, the primes have density $1$ and are evenly distributed between the reduced residue classes.

After the change of variables \(\beta=\alpha-a/q\), the main term from the arc around \(a/q\) becomes
\begin{align*}
    \frac{\mu(q)^2}{\varphi(q)^2}e\!\left(-\frac{aN}{q}\right)
    \int_{-R/(qN)}^{R/(qN)}
    \left(\sum_{n\leq N}e(\beta n)\right)^2 e(-N\beta)\,d\beta
    +O(\text{smaller terms}).
\end{align*}
Summing first over all reduced residues \(a \bmod q\) produces the Ramanujan sum \(c_q(-N)\), the same object that arose in the heuristic discussion. Since the arc half-width \(R/(qN)\) exceeds \(1/N\), we can use the sinc kernel decay \eqref{eq:geosum} to complete the \(\beta\)-integral to all of \([0,1)\) with negligible error, recovering the count of representations:
\begin{align*}
    \int_0^1  \left(\sum_{n\leq N}e(\beta n)\right)^2 e(-\beta N)\,d\beta = \sum_{n_1+n_2=N}1=N+O(1).
\end{align*}
Hence the total contribution of the major arcs is
\begin{align*}
   r_{\mathfrak{M}}(N)= N\sum_{q\leq R}\frac{\mu(q)^2}{\varphi(q)^2}c_q(-N)
    +O\!\left(NR^2 \exp(-c'\sqrt{\log N})\right),
\end{align*}
where again the truncated singular of \eqref{eq:truncated-SS} appears. Similarly as before, the tail converges absolutely and contributes negligibly, so
\begin{align*}
   r_{\mathfrak{M}}(N) = N\mathfrak{S}(N) + O\!\left(N\exp(-c''\sqrt{\log N})\right).
\end{align*}
Thus on the major arcs we recover unconditionally the Hardy--Littlewood main term for every individual \(N\).

\subsection{Vinogradov's bound}

It remains to handle the minor arcs. Recall that even under the assumption that $\Theta=1/2$ we were not able to extract the main term $N\mu(q)/\varphi(q)$ from $S(a/q + \beta)$ for all $q$, and without any $\Theta<1/2$ are left only with the limited range $q<(\log N)^A$. To see the right scale for what to expect, suppose heuristically that all  zeros of the Dirichlet $L$-functions modulo $q$ lie on the critical line, except possibly for one zero very close to $\Re s = 1$. From the explicit formula discussion above, the good  zeros on the critical line contribute at most
\begin{align*}
    q^{1/2}N^{1/2}(\log N)^{O(1)},
\end{align*}
while a single bad zero attached to a character $\chi$ could contribute
\begin{align*}
    \frac{N|\tau(\chi)|}{\varphi(q)}\leq \frac{N(\log N)^{O(1)}}{q^{1/2}}.
\end{align*}
Even a bad zero does not immediately ruin the minor arc bound, since the factor $q^{-1/2}$ still provides decay as $q$ grows. Vinogradov's achievement was to prove, without any hypothesis on the zeros, an upper bound of essentially this strength on the minor arc centred at $a/q$.

\begin{proposition}[Vinogradov--Vaughan]\label{prop:vino}
Let \(|\alpha-a/q|<q^{-2}\) with \((a,q)=1\). Then
\begin{align}\label{eq:vinogradov-minor}
    S(\alpha)
    \ll
    \Bigl(Nq^{-1/2}+N^{4/5}+(Nq)^{1/2}\Bigr)(\log N)^4.
\end{align}
\end{proposition}

If $\alpha$ lies on the minor arcs, a Dirichlet approximation $\alpha = a/q + \beta$ with $R < q \leq Q$ exists. Substituting this into \eqref{eq:vinogradov-minor} and recalling $Q = N/R = N(\log N)^{-A}$, one obtains a saving of a power of $\log N$ over the trivial bound $S(\alpha)\ll N$ throughout $\mathfrak{m}(R)$. With Vinogradov's minor arc bound we can estimate the mean square.

\begin{lemma}\label{lem:minor-arc-meansq}
For \(R<N^{1/2-\varepsilon}\) one has
\begin{align*}
    \sum_{N\leq X}|r_{\mathfrak{m}}(N)|^2
    \ll \bigl(X^2/R+X^{8/5}\bigr)X(\log X)^5.
\end{align*}
\end{lemma}
\begin{proof}
By Parseval's identity,
\[
\sum_{N\le X}|r_{\mathfrak{m}}(N)|^2
\le
\int_{\mathfrak{m}(R)}|S(\alpha)|^4\,d\alpha.
\]
For $\alpha\in\mathfrak{m}(R)$ there is a Dirichlet approximation $\alpha=a/q+\beta$ with $R<q\le X/R$ and $|\beta|\le q^{-2}$, so Proposition~\ref{prop:vino} gives
\[
S(\alpha)\ll \bigl(XR^{-1/2}+X^{4/5}\bigr)(\log X)^4.
\]
On the other hand, by Parseval and the Prime Number Theorem,
\[
\int_0^1|S(\alpha)|^2\,d\alpha=\sum_{n\leq X}\Lambda(n)^2 \leq (\log X) \sum_{n\leq X}\Lambda(n) \sim X\log X.
\]
Hence,
\begin{align*}
\int_{\mathfrak{m}(R)}|S(\alpha)|^4\,d\alpha
&\le
\Bigl(\sup_{\alpha\in\mathfrak{m}(R)}|S(\alpha)|^2\Bigr)
\int_0^1|S(\alpha)|^2\,d\alpha
\ll
\bigl(X^2/R+X^{8/5}\bigr)X(\log X)^5.
\end{align*}
This argument applies both an $L^\infty$ bound (on the minor arcs) and an $L^2$ bound (via Parseval globally), which is typical of the circle method.
\end{proof}
Combining major and minor arcs, we obtain the following unconditional Theorem that for the first time showed that almost all even integers are the sum of two primes.
\begin{theorem}[Chudakov; van der Corput; Estermann]\label{thm:CVE}
For any $A>0$ one has
\begin{align*}
     |\mathcal{E}(X)|\ll_A X (\log X)^{-A}.
\end{align*}
\end{theorem}

\begin{proof}
From the major arc analysis, $r_2(N) = N\mathfrak{S}(N) + O(N\exp(-c\sqrt{\log N})) + r_{\mathfrak{m}}(N)$ for every $N$, with $R = (\log N)^A$. The pointwise major arc error is negligible. By Lemma~\ref{lem:minor-arc-meansq} with $R = (\log X)^A$, one has $\sum_{N\leq X}|r_{\mathfrak{m}}(N)|^2 \ll X^3(\log X)^{-A+5}$. Since $\mathfrak{S}(N)\gg 1$ for even $N$, Chebyshev's inequality then shows that $r_2(N) \geq \frac{1}{2}N\mathfrak{S}(N)$ for all but $O_A(X(\log X)^{-A+5})$ even integers $N\leq X$. Replacing $A$ by $A+5$ gives the stated bound.
\end{proof}

This result was established independently by Chudakov \cite{C1937}, van der Corput \cite{vdC1937}, and Estermann \cite{E1938}.

We will not go into the proof of the crucial estimate \eqref{eq:vinogradov-minor} here, since it relies on exponential-sum methods that are somewhat orthogonal to the later developments in this survey. What matters to us is that, from this point on, the binary Goldbach problem is an unconditional almost-all theorem, and the subsequent works of Montgomery--Vaughan and Pintz are devoted to replacing the logarithmic saving in the exceptional set by a genuine power saving.

\section{Power saving}

Recall that $\mathcal{E}(X)$ denotes the set of even integers $N\leq X$ that cannot be written as a sum of two primes. Montgomery and Vaughan were the first to prove a genuine power saving for its size, in 1975 \cite{MV1975}.

\begin{theorem}[Montgomery--Vaughan]\label{thm:M-V}
There exists a \(\delta>0\) such that
\begin{align*}
    |\mathcal{E}(X)|\ll X^{1-\delta}.
\end{align*}
\end{theorem}

Their argument follows the general strategy of the previous section. One now chooses the major arc parameter $R$ to be a small power of $N$, say
\[
R=N^{\delta},
\]
so that Lemma~\ref{lem:minor-arc-meansq} leaves at most $O(X/R + X^{3/5})$ exceptional values from the minor arcs. The difficulty lies entirely on the major arcs: in the previous section the main term was recovered only for denominators $q\leq (\log N)^A$, whereas here one must work with all $q\leq R = N^\delta$, far beyond the Siegel--Walfisz range.

There are two new ingredients. The first is a way of controlling twisted prime sums for characters of power-sized modulus, which is provided by zero-density estimates. These allow one to work beyond the range of \eqref{eq:SW-char}, albeit with only a weak saving, and require that a possible Siegel zero be treated separately. The second ingredient is a way of organising the major arc expansion so that the $\sqrt{q}$-loss from the Gauss sum bound \eqref{eq:taubound} never appears term by term. This is where the generalised singular series enters.

\subsection{Beyond Siegel--Walfisz}

The key tool that allows one to pass from logarithmic to power-sized moduli is a zero-density estimate. Such estimates do not rule out the existence of bad  zeros close to the edge of the zero-free region; rather, they show that there cannot be too many of them. We will examine these estimates more closely in the next section, in the context of Pintz's refinements.

We require the notion of primitive Dirichlet character. We previously grouped Dirichlet characters into the principal character mod $q$, written $\chi_0^{(q)}(n)$, which is an indicator of $n$ being coprime to $q$, and all other characters. Given $r\mid q$ we can lift any character mod $r$ to a character mod $q$ by multiplying it with $\chi_0^{(q)}(n)$. Given $\chi$ mod $q$ we call the smallest $r$ from which it is lifted in this way its \emph{conductor} and call $\chi$ \emph{primitive} if $q=r$. Note that the principal character is the single character whose conductor is $1$. By $\sideset{}{^*}\sum_{\chi\bmod{r}}$ we denote a sum over primitive characters only.

We say that there is a \emph{Siegel zero of level $R$} if some primitive real character of conductor at most $R$ has a real zero $\widetilde\beta$ such that $\widetilde{\beta}>1-1/\log R$.

\begin{proposition}[Gallagher \cite{Gallagher1970}]\label{prop:gallagher}
There exists a constant $c>0$ such that the following holds for
\[
2\le \exp((\log N)^{1/2})\le R\le N.
\]
\begin{enumerate}
    \item If there is no Siegel zero of level \(R\), then
    \begin{align*}
        \sum_{r\le R}\ \ \sideset{}{^*}\sum_{\chi\bmod{r}}\max_{\substack{I\subset [1,N]\\ I\ \mathrm{interval}}}
\frac{1}{|I|+N/R}\left|\sum_{n\in I}\bigl(\Lambda(n)\chi(n)-\1_{\chi=\chi_0}\bigr)\right|
        \ll \exp\!\left(-c\frac{\log N}{\log R}\right).
    \end{align*}
    \item If there is a Siegel zero \(\widetilde\beta\) of level \(R\), attached to the exceptional primitive character \(\widetilde\chi\), then
    \begin{align*}
        \sum_{r\le R}\ \ \sideset{}{^*}\sum_{\chi\bmod{r}}
        \max_{\substack{I\subset [1,N]\\ I\ \mathrm{interval}}}
        \frac{1}{|I|+N/R}
        \left|
        \sum_{n\in I}
        \bigl(\Lambda(n)\chi(n)-\1_{\chi=\chi_0}+\1_{\chi=\widetilde\chi}n^{\widetilde\beta-1}\bigr)
        \right|\\ \ll
        (1-\widetilde\beta)(\log N)\exp\!\left(-c\frac{\log N}{\log R}\right).
    \end{align*}
\end{enumerate}
\end{proposition}

The saving comes from the factor $\exp(-c\log N/\log R)$, which can be made smaller than any fixed constant by taking $R = N^\delta$ with $\delta$ sufficiently small. It cannot, however, in this range compensate for any loss of even logarithmic size. This is the weak saving mentioned above. It is nevertheless sufficient once the major arc contribution is reorganised to avoid the $\sqrt{q}$-loss from \eqref{eq:taubound}. Within the second part of the statement is hidden a remarkable fact: If there is a Siegel zero, it pushes all other  zeros further away. Even more, the closer the Siegel zero is to $1$, the stronger this effect. For this reason the error term improves by a factor 
$(1-\widetilde{\beta})(\log N)$.

Observe further that this result implies nothing for intervals of shorter size, if $|I| < N/R^3$  the bound is trivial. This is necessary because the saving is not strong enough to express a short interval as a difference of two long ones starting from $1$, as was done in \eqref{eq:sum-by-parts}.

\subsection{Major arcs}

We return to the major arc contribution
\[
r_{\mathfrak{M}}(N)=\int_{\mathfrak{M}(R)} S(\alpha)^2e(-\alpha N)\,d\alpha.
\]
Writing $\alpha=a/q+\beta$ and using the orthogonality relation \eqref{eq:orth1} exactly as before, then reducing to primitive characters $\chi_i$ of conductors $r_i\mid q$, one arrives at the decomposition
\begin{align*}
    r_{\mathfrak{M}}(N)
    =
    N\mathfrak{S}(N)
    +
    E_{\mathfrak{M}}(N)+\textup{mixed terms}+O(N^{1-\delta}),
\end{align*}
where the main term comes from the poles at $s=1$ from the pair of principal characters ($r_1=r_2=1$) and
\begin{align}\label{eq:EM-def}
    E_{\mathfrak{M}}(N):=
    \sideset{}{^*}\sum_{\substack{r_1,r_2\le R\\ \chi_i \,(\mathrm{mod}\, r_i)}}
    \sum_{\substack{q\le R\\ [r_1,r_2]\mid q}}
    \frac{\tau(\overline{\chi_1\chi_0^{(q)}})\tau(\overline{\chi_2\chi_0^{(q)}})}{\varphi(q)^2}
    c_{\chi_1\chi_2\chi_0^{(q)}}(N)
    \int_{|\beta|<R/(qN)} R_{\chi_1}(\beta)R_{\chi_2}(\beta)e(-\beta N)\,d\beta,
\end{align}
with
\begin{align*}
    R_\chi(\beta):=\sum_{n\le N}\bigl(\Lambda(n)\chi(n)-\1_{\chi=\chi_0}\bigr)e(\beta n).
\end{align*}
Note that for simplicity we are ignoring mixed terms in which one $S$ contributes a main term and the other an error, these are strictly easier to handle. 

This quantity $E_{\mathfrak{M}}(N)$ will be a recurring object in the remainder of the survey. The larger $R$ can be made while keeping $E_{\mathfrak{M}}(N)$ smaller than the main term, the better the saving in the exceptional set. The characters encode the congruence conditions, while the $\beta$-integral reflects the size condition in the spirit of \eqref{eq:circlemethod0}. The sum over $q$ may be thought of as the initial segment of a generalised singular series associated to the pair $(\chi_1,\chi_2)$. The error $E_{\mathfrak{M}}(N)$ is precisely the contribution that bad zeros, through the correction terms \eqref{eq:correctionterms}, make to the binary Goldbach problem on the major arcs.

In the following sections three different strategies for treating $E_{\mathfrak{M}}(N)$ will appear. First, the Montgomery--Vaughan approach, which takes absolute values everywhere and applies Gallagher's lemma below; no cancellation in the sum over $q$ is captured. Second, Pintz's refinement, which extracts the contribution of so-called generalised exceptional  zeros as explicit secondary main terms, applying Gallagher's lemma only to the remainder. Third, in Section~7 we replace Gallagher's lemma altogether by a smooth major arc weight, producing a fully explicit formula.

A key tool for controlling the $\beta$-integral is the following lemma.
\begin{lemma}[Gallagher]\label{lem:Gall}
For any complex numbers \(\alpha_n\) and any \(\eta>0\),
\begin{align*}
    \int_{|\beta|\leq \eta}\left|\sum_{n\leq N}\alpha_n e(\beta n)\right|^2\,d\beta
    \ll
    \int_0^{2N}
    \eta\left|\sum_{x-\eta^{-1}<n\leq x}\alpha_n\right|^2\,dx.
\end{align*}
\end{lemma}

\subsection{Montgomery--Vaughan}

The simplest way to treat \(E_{\mathfrak{M}}(N)\) is to apply the triangle inequality and enlarge the range of integration:
\begin{align*}
    |E_{\mathfrak{M}}(N)|
    \le
    \sideset{}{^*}\sum_{\substack{r_1,r_2\le R\\ \chi_i \,(\mathrm{mod}\, r_i)}}
    \sum_{\substack{q\le R\\ [r_1,r_2]\mid q}}
    \frac{|\tau(\overline{\chi_1\chi_0^{(q)}})\tau(\overline{\chi_2\chi_0^{(q)}})|}{\varphi(q)^2}
    \bigl|c_{\chi_1\chi_2\chi_0^{(q)}}(N)\bigr|
    \int_{|\beta|<R/N}|R_{\chi_1}(\beta)R_{\chi_2}(\beta)|\,d\beta.
\end{align*}

As shown in \cite[Lemma~5.5]{MV1975},
\begin{align}\label{eq:MV-L55}
    \sum_{\substack{q\ge 1\\ [r_1,r_2]\mid q}}
    \frac{|\tau(\overline{\chi_1\chi_0^{(q)}})\tau(\overline{\chi_2\chi_0^{(q)}})|}{\varphi(q)^2}
    \bigl|c_{\chi_1\chi_2\chi_0^{(q)}}(N)\bigr|
    \ll \mathfrak{S}(N).
\end{align}
In other words, the absolute values of the new perturbation terms are collectively bounded by the classical singular series.

An application of Cauchy--Schwarz followed by Lemma~\ref{lem:Gall} (with $\eta = R/N$) gives
\begin{align}
 \nonumber
    \left(
    \int_{|\beta|<R/N}|R_{\chi}(\beta)|^2\,d\beta
    \right)^{1/2}
    &\ll
    \left(
    \int_0^{2N}
    \frac{R}{N}
    \left|
    \sum_{x-N/R<n\leq x}\bigl(\Lambda(n)\chi(n)-\1_{\chi=\chi_0}\bigr)
    \right|^2
    dx
    \right)^{1/2} \\
    \label{eq:maxint}
    &\ll
    N^{1/2}
    \max_{x<2N}
    \frac{R}{N}
    \left|
    \sum_{x-N/R<n\leq x}\bigl(\Lambda(n)\chi(n)-\1_{\chi=\chi_0}\bigr)
    \right|,
\end{align}
where the second step bounds the $L^2$ integral over $[0,2N]$ by $\sqrt{2N}$ times the $L^\infty$ norm. Putting everything together yields
\begin{align*}
     |E_{\mathfrak{M}}(N)|
     \ll
     N\mathfrak{S}(N)
     \left(
     \sideset{}{^*} \sum_{\substack{r\le R\\ \ \chi \bmod{r}}}
     \max_{x<2N}
     \frac{R}{N}
     \left|
     \sum_{x-N/R<n\leq x}\bigl(\Lambda(n)\chi(n)-\1_{\chi=\chi_0}\bigr)
     \right|
     \right)^2.
\end{align*}
An application of Proposition~\ref{prop:gallagher} then bounds this by
\begin{align*}
    N\mathfrak{S}(N)\exp\!\left(-c\frac{\log N}{\log R}\right),
\end{align*}
which is smaller than $N\mathfrak{S}(N)/2$ if $R$ is chosen as a sufficiently small power of $N$.

We briefly mention the additional complication when a Siegel zero exists. In that case the exceptional character cannot be absorbed into the error term. Instead it produces a genuine secondary main term
\begin{align}\label{eq:MVsiegelsecondary}
    \mathfrak{S}(N,\widetilde\chi,\widetilde\chi)
    \sum_{n_1+n_2=N}n_1^{\widetilde\beta-1}n_2^{\widetilde\beta-1}
    =
    \mathfrak{S}(N,\widetilde\chi,\widetilde\chi)\,
    \frac{\Gamma(\widetilde\beta)^2}{\Gamma(2\widetilde\beta)}\,
    N^{2\widetilde\beta-1},
\end{align}
which can be comparable in size to the main term. (We have expressed it in closed form using the beta integral; the same quantity will appear later as a special case of the archimedean factor in Section~7.) More precisely, for certain $N$ the best available lower bound takes the form
\begin{align*}
    N\mathfrak{S}(N)
    -
    \mathfrak{S}(N,\widetilde\chi,\widetilde\chi)\,
    \frac{\Gamma(\widetilde\beta)^2}{\Gamma(2\widetilde\beta)}\,
    N^{2\widetilde\beta-1}
    \gg
    N(1-\widetilde\beta)(\log N)\mathfrak{S}(N).
\end{align*}
This is still positive, and the second part of Proposition~\ref{prop:gallagher} supplies exactly the extra factor $(1-\widetilde\beta)\log N$ needed to compensate for this loss. Thus Theorem~\ref{thm:M-V} holds in the Siegel-zero case as well.

\section{Pintz's bound}

Gallagher's Prime number Theorem, Proposition~\ref{prop:gallagher}, which was the key to
Montgomery--Vaughan's power saving, relies internally on zero-density
estimates.  
Pintz's approach improves the Montgomery--Vaughan argument (and subsequent works making $\delta$ explicit) by
bringing more precise structure to the major arc term: rather than taking
absolute values from the start, the contribution of zeros near $1$ is
extracted as an explicit finite collection of secondary main terms, and the
arithmetic of the resulting generalised singular series is then exploited
to show that these terms cannot collectively dominate the main term.

Throughout this section we assume that no Siegel zero exists.
Two considerations motivate this.
First, already in Pintz's work the Siegel zero case yields numerically
stronger bounds for the exceptional set, since the Deuring--Heilbronn
phenomenon forces all other $L$-functions to have their zeros away from
$1$.
Second, more recent work of Matomäki--Merikoski \cite{MM2023}
shows that the existence of a Siegel zero gives substantially more
information about Goldbach's problem than is available when it is absent.

\subsection{Zero density estimate applications}

The two results in this subsection illustrate, in relatively clean settings,
what zero-density estimates can achieve and are both relevant in Pintz's work. Theorem~\ref{thm:linnik}
(Linnik's theorem) is the prototype for the quasi-diagonal structure that
Pintz exploits in \cite{P2023}, while Theorem~\ref{thm:primealmostall}
(primes in almost all short intervals) is related to the treatment of
non-generalised-exceptional  zeros in \cite{P2018}.

We use the standard notation for the zero-counting function of a
Dirichlet $L$-function: for a character $\chi$ and real $\alpha,T>0$,
\[
N(\alpha,T,\chi)
:=
\#\{\rho=\beta+i\gamma:\ L(\rho,\chi)=0,\ \beta\ge\alpha,\ |\gamma|\le T\}.
\]
A bound of the shape 
\begin{align}\label{eq:lfzde}
\sum_{\chi\bmod{q}}N(\alpha,T,\chi)\ll(qT)^{A(1-\alpha)}
\end{align} is called $\log$-free zero density estimate and
controls how many zeros, averaged over characters modulo $q$, can lie
close to $1$, where the exponent $A$ measures the quality of the estimate.
We similarly write $N(\sigma,T)$ for the zero-counting function of
$\zeta(s)$ itself.

\begin{theorem}[Linnik \cite{Linnik1944}]\label{thm:linnik}
There exists a constant $L$ such that the least prime congruent to
$a\bmod{q}$ with $(a,q)=1$ is $O(q^L)$.
\end{theorem}

\begin{proof}[Sketch]
Using the orthogonality of characters and the explicit formula one obtains
\[
\sum_{\substack{n\leq X\\ n\equiv a\bmod{q}}}\Lambda(n)
=
\frac{X}{\varphi(q)}
-
\frac{1}{\varphi(q)}
\sum_{\chi\bmod{q}}\overline{\chi(a)}
\sum_{\rho_\chi}\frac{X^{\rho_\chi}}{\rho_\chi}
+
O(\text{lower-order terms}).
\]
Set $X=q^L$. For a zero $\rho=\beta+i\gamma$, write $\lambda_\rho:=(1-\beta)\log q$,
so that $X^{\beta-1}=e^{-L\lambda_\rho}$.
After a dyadic decomposition in $|\gamma|$, the problem reduces to bounding
the weighted zero sum
\[
\sum_{\chi\bmod{q}}\sum_{\rho_\chi}e^{-L\lambda_{\rho_\chi}},
\]
up to polynomial and logarithmic factors in $q$.
The zero-density bound
$\sum_{\chi\bmod{q}}N(\alpha,T,\chi)\ll(qT)^{A(1-\alpha)}$
gives, after integrating against $e^{-L\lambda}$ over $\lambda>0$,
a bound of $O(q^{C-L/A})$ for some absolute constant $C$.
Choosing $L>AC$ ensures that the total zero contribution is $o(X/\varphi(q))$,
so the sum over $\Lambda$ is positive and a prime $p\equiv a\bmod{q}$ with
$p\ll q^L$ exists.
\end{proof}

For the next result we write, for $H>0$,
\[
V(X,H)
:=
\int_X^{2X}
\left|\sum_{x<n\le x+H}\Lambda(n)-H\right|^2dx.
\]
$V(X,H)=o(H^2X)$ means that for almost all $x\in[X,2X]$ the prime sum
$\sum_{x<n\le x+H}\Lambda(n)$ is asymptotic to $H$, i.e.\ primes are
equidistributed in almost all short intervals of length $H$.

\begin{theorem}\label{thm:primealmostall}
Suppose that
\[
N(\sigma,T)\ll T^{A(1-\sigma)}(\log T)^B
\]
for some constants $A,B>0$. Then $V(X,H)=o(H^2X)$ for every
$H=X^\theta$ with $\theta>1-2/A$.
\end{theorem}

\begin{proof}[Sketch]
By the explicit formula for $\psi(x)$, the integrand in $V(X,H)$ is
\[
\psi(x+H) - \psi(x) - H = -\sum_{|\gamma|\le T}\frac{(x+H)^\rho-x^\rho}{\rho} + O\left(\frac{x\log^2 x}{T}\right).
\]
For $|\gamma| \le T$ and $T \approx X/H$, we use the approximation $\frac{(x+H)^\rho-x^\rho}{\rho} \approx Hx^{\rho-1}$. Squaring and integrating over $x \in [X, 2X]$, and assuming off-diagonal terms contribute negligibly, we have
\[
V(X,H) \ll H^2 \int_X^{2X} \left| \sum_{|\gamma| \le T} x^{\rho-1} \right|^2 dx + \frac{X^3 \log^4 X}{T^2}.
\]
The integral is bounded by $X \sum_{|\gamma| \le T} X^{2\beta-2}$. Using the density estimate $N(\sigma, T)$, this becomes
\[
V(X,H) \ll H^2 X \log X \int_{1/2}^1 X^{2\sigma-2} T^{A(1-\sigma)} d\sigma + \frac{X^3 \log^4 X}{T^2}.
\]
Let $\sigma = 1 - \lambda$ and $T = X/H$. The integrand becomes $(X^{A-2}/H^A)^\lambda$. This decays geometrically in $\lambda$ provided $H^A > X^{A-2}$, which is equivalent to $H > X^{1-2/A}$ (i.e., $\theta > 1-2/A$).

In this range, the integral is $O(1/\log(H^A/X^{A-2})) = o(1)$, so the first term is $o(H^2 X)$. To ensure the tail $X^3/T^2$ is also $o(H^2 X)$, we choose $T$ slightly larger than $X/H$, say $T = (X/H)\log^C X$. This maintains the geometric decay while suppressing the error term.
\end{proof}

We remark that already for $\zeta(s)$ zero density estimates are an active field of study, see for example the recent breakthrough of Guth and Maynard \cite{GuthMaynard2024} that allows the choice $A=30/13$. Their result, $N(\sigma, T) \ll T^{\frac{30}{13}(1-\sigma)+o(1)}$, represents the first major improvement to the exponent since Huxley's $12/5$ in 1972, further narrowing the range where $V(X,H)$ might fail, to be $o(H^2X)$.

\subsection{Pintz}
\label{sec:pintz-test}

We now outline the argument in \cite{P2018} and \cite{P2023}, recalling that
we only consider the case where no Siegel zero exists.
The two papers have rather different r\^oles.

In the Montgomery--Vaughan treatment, the term $E_{\mathfrak M}(N)$ defined
in \eqref{eq:EM-def} was bounded by taking absolute values throughout before
applying Gallagher's lemma, thereby losing all cancellation between different $q$ and treating all  zeros as equally dangerous.
Pintz's first paper \cite{P2023}  separates the zeros into those very
close to $1$,  described  as \emph{generalised exceptional} (the corresponding conductors bear the same name) and the rest.
The former are treated by running the explicit formula directly on the major
arc integrals, making them appear as a finite collection of explicit secondary
main terms, each with a coefficient given by a generalised singular series.
This is precisely how the exceptional Siegel zero gave a secondary main term
in the Montgomery--Vaughan argument \eqref{eq:MVsiegelsecondary}. The new
point is that the same structure is imposed on \emph{all}  zeros very close real part $1$,
not just those coming from a single exceptional real character.
Pintz's second paper \cite{P2023} then studies these new terms arithmetically
and exploits the fact that only a very restricted class of character pairs can
contribute substantially.
In this way the problem is reduced to a quasi-diagonal sum over zeros, much
closer in spirit to questions around Linnik's theorem
(Theorem~\ref{thm:linnik}) than to the treatment of
Montgomery--Vaughan.

We now describe the first step in more detail.
Recall from \eqref{eq:EM-def} that the quantity $E_{\mathfrak M}(N)$ is
already normalised so that the main term has been removed.
The remaining double sum in \eqref{eq:EM-def} is entirely governed by the  zeros of the Dirichlet $L$-functions involved.

To single out the most dangerous of these zeros, we define, for large
parameters $H$ and $T$ to be chosen later as sufficiently large constants
depending on $\varepsilon$,
\begin{align}\label{eq:EHT-def}
    \mathcal{E} = \mathcal{E}(H,T,R,X)
    :=
    \bigl\{
        (\varrho,\chi) :
        \chi \text{ primitive}, \ \mathrm{cond}(\chi) \le R,\
        L(\varrho,\chi)=0,\ \\ 
        \Re\varrho \ge 1 - \tfrac{H}{\log X},\
        |\Im\varrho| \le T
    \bigr\}.
\end{align}
By a log-free zero-density estimate \eqref{eq:lfzde}, the cardinality
satisfies $|\mathcal{E}| \le Ce^{2H}$.
This is the crucial feature: for fixed $H$ the set $\mathcal{E}$ is
\emph{finite}, so the explicit formula below is a finite sum over bad zeros .

Before stating the theorem, we introduce the generalised singular series
associated to a pair of primitive characters $\chi_1 \bmod{r_1}$ and
$\chi_2\bmod{r_2}$.
Recall that in \eqref{eq:EM-def}, the $q$-sum was truncated at $q \le R$.
Completing this sum to infinity defines
\begin{align}\label{eq:genSS-def}
    \mathfrak{S}(\chi_1,\chi_2,N)
    :=
    \sum_{\substack{q=1\\ [r_1,r_2]\mid q}}^{\infty}
    \frac{\tau(\overline{\chi_1\chi_0^{(q)}})\,\tau(\overline{\chi_2\chi_0^{(q)}})}
         {\varphi(q)^2}\,
    c_{\chi_1\chi_2\chi_0^{(q)}}(N).
\end{align}
When $\chi_1 = \chi_2 = \chi_0^{(1)}$ this reduces to $\mathfrak{S}(N)$,
the classical Hardy--Littlewood singular series.
For a general pair it encodes the same local congruence information as
$\mathfrak{S}(N)$, but weighted by the arithmetic of $\chi_1$ and $\chi_2$.
A key bound, proved in the Main Lemma of \cite{P2018}, is
\begin{align}\label{eq:genSS-bound}
    |\mathfrak{S}(\chi_1,\chi_2,N)|\le \mathfrak{S}(N),
\end{align}
so the generalised singular series is never larger in absolute value than
the classical one. Note that in particular, there is no loss of a constant, as there had been in \eqref{eq:MV-L55}.

More importantly, for any $\eta>0$ small enough, there exists an integer $C(\eta)$ such that
\begin{align}\label{eq:genSS-small}
    |\mathfrak{S}(\chi_1,\chi_2,N)|\le \eta
\end{align}
unless all three of the following divisibility conditions hold:
\begin{align}\label{eq:genSS-conditions}
    r_1 \mid C(\eta)\,N,\qquad
    r_2 \mid C(\eta)\,N,\qquad
    \mathrm{cond}(\chi_1\overline{\chi_2}) < \eta^{-3}.
\end{align}
In other words, the generalised singular series is small unless the conductors
$r_1, r_2$ both divide a bounded multiple of $N$ and the product character
$\chi_1\overline{\chi_2}$ has small conductor.
Most character pairs therefore contribute negligibly, and the effective sum in the explicit formula is sparser than it appears.

\begin{theorem}[Pintz, explicit formula for the major arcs
{\cite[Thm.~1]{P2018}, \cite[Thm.~A]{P2023}}]\label{thm:pintz-explicit}
Let $0<\varepsilon<\vartheta < 4/9-\varepsilon$ and $R = X^\vartheta$.
Then for every even $N\in[X/2,X]$,
\begin{align}\label{eq:pintz-EF}
    E_{\mathfrak M}(N)
    =
    \sum_{\substack{(\varrho_1,\chi_1)\in\mathcal{E}\\(\varrho_2,\chi_2)\in\mathcal{E}}}
    \mathfrak{S}(\chi_1,\chi_2,N)\,
    \frac{\Gamma(\varrho_1)\Gamma(\varrho_2)}{\Gamma(\varrho_1+\varrho_2)}\,
    N^{\varrho_1+\varrho_2-1}
    +
    O\!\left(Xe^{-cH}+\frac{X}{\sqrt{T}}+X^{1-\varepsilon}\right).
\end{align}
The generalised singular series $\mathfrak{S}(\chi_1,\chi_2,N)$ satisfies
\eqref{eq:genSS-bound}--\eqref{eq:genSS-conditions}.
\end{theorem}

Note that
the beta-type factor satisfies
\begin{align}\label{eq:beta-bound}
    \left|\frac{\Gamma(\varrho_1)\Gamma(\varrho_2)}{\Gamma(\varrho_1+\varrho_2)}\right|
    =
    |B(\varrho_1,\varrho_2)|
    \le
    B(\Re\varrho_1,\Re\varrho_2)
    \le 1 + O\!\left(\tfrac{1}{\log X}\right)
\end{align}
for zeros near $1$, by the integral representation of the beta function.
Together with \eqref{eq:genSS-bound}, this means each term is bounded in
absolute value by $(1+o(1))\mathfrak{S}(N)N^{\Re\varrho_1+\Re\varrho_2-1}$,
and whether the sum is small depends entirely on how many  zeros can be close
to $1$ simultaneously.

\medskip

The second step, carried out in \cite{P2023}, is to exploit the sparsity
coming from \eqref{eq:genSS-conditions} to prove $r_{\mathfrak M}(N)>0$ for
all but $X^{1-\vartheta}$ even integers $N$.
Choosing $\eta$ small enough relative to $\varepsilon$, one sees that the
contribution to \eqref{eq:pintz-EF} from pairs not satisfying
\eqref{eq:genSS-conditions} is at most $\varepsilon N\mathfrak{S}(N)$.
The question therefore reduces to showing
\begin{align}\label{eq:pintz-positivity}
    \sideset{}{'}\sum_{(\varrho_1,\chi_1),(\varrho_2,\chi_2)\in\mathcal{E}}
    N^{\Re\varrho_1 + \Re\varrho_2 - 2}
    < 1 - 2\varepsilon,
\end{align}
where $\sum'$ denotes the restriction to pairs satisfying
\eqref{eq:genSS-conditions}.

The condition \eqref{eq:genSS-conditions} is the key structural feature.
Since the conductors $r_1$ and $r_2$ both divide $C_1(\varepsilon)N$, for
each fixed $N$ the contributing  zeros belong to $L$-functions whose
conductors divide a common bounded multiple of $N$.
Pintz exploits this by partitioning the even integers $m \in [X/2,X]$ into
at most $2^K$ classes $\mathcal{M}(R')$, where $R' \subseteq \mathcal{E}$ is
the subset of generalised exceptional characters whose conductors divide
$C_1(\varepsilon)N$.
There are at most $K \le Ce^{2H}$ generalised exceptional characters, so the
number of classes is bounded by a constant depending on $\varepsilon$.
Classes for which $q(R') := \mathrm{lcm}(\mathrm{cond}(\chi) : \chi \in R') > X^\vartheta$
contain at most $C_1(\varepsilon)X^{1-\vartheta}$ integers and can be
discarded as part of the exceptional set.

For the remaining classes, fix one with $q := q(R') \le X^\vartheta$.
All contributing zeros now belong to $L$-functions of conductors dividing
$q$, so the sum \eqref{eq:pintz-positivity} becomes
\begin{align}\label{eq:pintz-S0}
    S_0 = \sideset{}{''}\sum_{\varrho_1,\varrho_2}
    q^{-(1/\vartheta)(\delta_1 + \delta_2)} < 1 - 2\varepsilon,
    \qquad \delta_k = 1 - \Re\varrho_k,
\end{align}
where $\sum''$ is now over  zeros with $r_1 \mid q$, $r_2 \mid q$, and
$\mathrm{cond}(\chi_1\overline{\chi_2}) < C_0(\varepsilon)$.
The problem has been reduced from one involving zeros
 of $L$-functions of different moduli to one where all relevant zeros
belong to characters modulo a single $q \le X^\vartheta$, which strongly
resembles the estimation of Linnik's constant.

A further simplification comes from an equivalence relation on the
generalised exceptional characters: $\chi \sim \chi'$ if there is a chain
$\chi = \chi_1, \ldots, \chi_n = \chi'$ with
$\mathrm{cond}(\chi_\nu\overline{\chi_{\nu+1}}) < C_0(\varepsilon)$ for each
consecutive pair.
Equivalent characters satisfy $\mathrm{cond}(\chi\overline{\chi'}) <
C_3(\varepsilon)$, and since $\delta \gg (\sqrt{q}\log^2 q)^{-1}$ by
Siegel's theorem, no generalised exceptional zero is equivalent to the
trivial character, so $S_0$ contains only genuine zero pairs.
Distributing the zeros among their equivalence classes
$\mathcal{H}_\nu$ ($\nu = 1, \ldots, M \le K$), one obtains
\begin{align}\label{eq:Snu-bound}
    S_0 \le S := \sum_{\nu=1}^M S_\nu^2,
    \qquad
    S_\nu := \sum_{\substack{\varrho \in \mathcal{E},\, \chi \in \mathcal{H}_\nu}}
    q^{-(1/\vartheta)\delta},
\end{align}
reducing the problem to $M \le K$ independent one-dimensional Linnik-type
estimates.

These are then handled using the three principles from Heath-Brown's work
\cite{HB1992} on Linnik's constant: the classical zero-free region for
$\prod_{\chi\,(\mathrm{mod}\,q)}L(s,\chi)$, the Deuring--Heilbronn
phenomenon, and log-free zero-density estimates.
The standard density estimates used in earlier works on the exceptional set
are not strong enough here, since they count only the number of
$L$-functions with at least one zero in a given range rather than the total
number of zeros.
Pintz therefore introduces a new log-free density theorem (his Theorem~C)
that bounds the total number of zeros, losing only a slight constant
compared to Heath-Brown's Lemma~11.1.
A further ingredient is a specialised greedy algorithm (his Theorem~K) that
yields sharper numerical bounds for the weighted zero sums than those
obtainable by standard partial summation.
Combining these tools gives the following.

\begin{proposition}[Pintz {\cite[Thm.~1]{P2023}}]\label{prop:pintz-main}
There exists $\varepsilon > 0$ such that for $\vartheta < 0.28$ and all
sufficiently large $X$ (with an ineffective constant),
\begin{align*}
    S_0 < 1 - \varepsilon.
\end{align*}
\end{proposition}

Together with the minor arc bound, which contributes at most $O(X/R)$
exceptional values, this gives $|\mathcal{E}(X)| < X^{0.72}$.

\section{A fully explicit formula}
\label{sec:fef}
 
The two preceding sections treated the major-arc error $E_{\mathfrak M}(N)$
defined in \eqref{eq:EM-def} in two rather different ways.
Montgomery--Vaughan applied Gallagher's lemma after taking absolute values
throughout, losing all structural information about the underlying zeros.
Pintz retained finitely many zeros explicitly as secondary main terms,
estimating the remainder by absolute values.
The aim of this section is to keep \emph{every} zero visible: replacing
the sharp characteristic function of the major arcs by a smooth weight
allows one to apply the explicit formula to each twisted prime sum directly,
yielding a formula in which all zeros enter explicitly.
 
Throughout this section we work with a smoothly truncated exponential sum.
Fix $\phi\in C_c^\infty(1/5,4/5)$, $\phi\geq 0$, and 
normalised so that
\begin{align}\label{eq:phi-norm}
\int_0^1 \phi(t)\phi(1-t)\,dt = 1,
\end{align}
and define
\begin{align}\label{eq:Sphi-def}
S_\phi(\alpha) := \sum_{n} \Lambda(n)\,\phi(n/N)\,e(\alpha n).
\end{align}
We define the weighted convolution sum
\begin{align*}
    r_\phi(N)=\sum_{n_1+n_2=N} \Lambda(n_1)\phi(n_1/N)\Lambda(n_2)\phi(n_2/N).
\end{align*}
Since $\phi$ is non-negative and bounded, if we can show that $ r_\phi(N)$ is large, we can deduce that $N$ is the sum of two primes. 
 
\subsection{A smooth major-arc weight}
 
\begin{definition}\label{def:bR-fef}
Let $G\in C_c^\infty(\R)$ be a fixed real-valued, non-negative function
supported in $[-2,2]$, equal to a positive constant on $[-1,1]$, and normalised by
$\int_{\R}G(t)\,dt=1$. For $T\ge 1$ set $G_T(x):=T^{-1}G(x/T)$.
Fix $R\ge 1$ and scales $T_q\ge 1$ for each $q\le R$. Define
\begin{align}\label{eq:bR-def}
    b_R(n):=\sum_{q\le R}c_q(n)\,G_{T_q}(n).
\end{align}
\end{definition}
 
The discrete Fourier transform $\widehat{b_R}(\alpha):=\sum_{n\in\Z}b_R(n)e(-\alpha n)$
satisfies, 

via $c_q(n)=\sum^*_{a\bmod q}e_q(an)$,
\begin{align}\label{eq:bR-FT}
\widehat{b_R}(\alpha)
=\sum_{q\le R}\ \ \sideset{}{^*}\sum_{a\bmod q}\widehat{G_{T_q}}(\alpha-a/q),
\end{align}
where $\widehat{G_T}(\beta):=\sum_{n\in\Z}G_T(n)e(\beta n)$.
 
\begin{lemma}\label{lem:bR-FT}
Let $T_q=\eta Nq/R$ for some $0<\eta\le 1$, and assume $10R^2\le\eta N$.
Then
\begin{align}\label{eq:bR-FT-values}
\widehat{b_R}(\alpha)
=
\begin{cases}
1+O(\eta^2)+O_A\!\left(R^2(\eta N/R^2)^{-A}\right),& \alpha\in\mathfrak M(R),\\[0.5ex]
O(1)+O_A\!\left(R^2(\eta N/R^2)^{-A}\right),& \alpha\in\mathfrak m(R).
\end{cases}
\end{align}
\end{lemma}
 
\begin{proof}
By Poisson, $\widehat{G_T}(\beta)=\widehat{G}_\R(T\beta)+O_A(T^{-A})$,
with $\widehat{G}_\R$ the continuous Fourier transform of $G$, assumed to be even.
Repeated integration by parts and Taylor expansion give
\begin{align}\label{eq:GT-decay}
|\widehat{G_T}(\beta)|&\ll_A (1+T\|\beta\|)^{-A}+T^{-A},\\
\label{eq:GT-taylor}
\widehat{G_T}(\beta)&=1+O\bigl((T\|\beta\|)^2\bigr)+O_A(T^{-A}),
\quad \text{ for } T\|\beta\|\le 1,
\end{align}
where $\|\beta\|$ is the distance to the nearest integer.
Suppose $\alpha=a_0/q_0+\beta\in\mathfrak M(R)$ with $|\beta|\le R/(q_0 N)$.
The term $(q,a)=(q_0,a_0)$ satisfies $T_{q_0}|\beta|\le\eta$,
so \eqref{eq:GT-taylor} gives $\widehat{G_{T_{q_0}}}(\beta)=1+O(\eta^2)$.
For every other reduced pair $(q,a)\ne(q_0,a_0)$ with $q\le R$,
the Farey separation $\|\alpha-a/q\|\gg 1/(qq_0)$ and $10R^2\le\eta N$
give $T_q\|\alpha-a/q\|\gg\eta N/R^2$; with $O(R^2)$ such pairs,
\eqref{eq:GT-decay} yields the stated error.
On the minor arcs $T_q\|\alpha-a/q\|>\eta$ for all $a/q$, so the
closest term is $O(1)$ and the rest give the same additive error.
\end{proof}
 
\subsection{Second-moment approximation}
 
\begin{lemma}\label{lem:bR-approx}
Let $R=X^\vartheta$ with $0<\vartheta<1/2$ and $10R^2\le\eta N$.  Then
\begin{align}\label{eq:bR-approx}
\sum_{N\le X}\Bigl|r_\phi(N)-\int_0^1 S_\phi(\alpha)^2\,
\widehat{b_R}(\alpha)\,e(-N\alpha)\,d\alpha\Bigr|^2
\ll
\bigl(R^{-1}+X^{-2/5}+\eta^4\bigr)X^3(\log X)^5.
\end{align}
\end{lemma}
 
\begin{proof}
 It suffices to bound
$\sum_{N\le X}|\int_0^1 S_\phi (\alpha)^2(1-\widehat{b_R}(\alpha))e(-N\alpha)\,d\alpha|^2$.
Parseval gives
\[
\sum_{N\le X}\Bigl|\int_0^1
S_\phi(\alpha)^2(1-\widehat{b_R}(\alpha))e(-N\alpha)\,d\alpha\Bigr|^2
\le\int_0^1|S_\phi(\alpha)|^4\,|1-\widehat{b_R}(\alpha)|^2\,d\alpha.
\]
On $\mathfrak M(R)$, $|1-\widehat{b_R}(\alpha)|=O(\eta^2)$ by Lemma~\ref{lem:bR-FT},
so using the bounds $|S_\phi(\alpha)|\le X$ and $\int_0^1|S_\phi(\alpha)|^2\ll X\log X$,
\[
\int_{\mathfrak M(R)}|S_\phi(\alpha)|^4|1-\widehat{b_R}(\alpha)|^2\,d\alpha\ll\eta^4 X^3\log X.
\]
On $\mathfrak m(R)$, $|1-\widehat{b_R}(\alpha)|=O(1)$ and $S_\phi$ satisfies the
same Vinogradov--Vaughan minor-arc estimate as $S$, giving
$\int_{\mathfrak m(R)}|S_\phi(\alpha)|^4\,d\alpha\ll(X^3/R+X^{13/5})(\log X)^5$.
Combining proves \eqref{eq:bR-approx}.
\end{proof}
 
Optimising with $\eta=R^{-1/4}$, valid for $R\le(X/10)^{4/9}$, and
fixing $R=X^\vartheta$ with $0<\vartheta<4/9$, \eqref{eq:bR-approx} becomes
\begin{align}\label{eq:bR-approx-clean}
\sum_{N\le X}\Bigl|r_\phi(N)-\int_0^1
S_\phi^2(\alpha)\widehat{b_R}(\alpha)e(-N\alpha)\,d\alpha\Bigr|^2
\ll\bigl(X^{3-\vartheta}+X^{13/5}\bigr)(\log X)^5.
\end{align}
 
\subsection{Character expansion and the explicit formula}
 
Expanding the smoothed integral via the character sum structure of $b_R$
gives, by the same computation as in \eqref{eq:EM-def},
\begin{align}\label{eq:bR-character}
\int_0^1 S_\phi^2(\alpha)\widehat{b_R}(\alpha)e(-N\alpha)\,d\alpha
=\sideset{}{^*}\sum_{\substack{r_1,r_2\le R\\\chi_i\,(\mathrm{mod}\,r_i)}}
\sum_{\substack{q\le R\\{[r_1,r_2]}\mid q}}
A_q(N;\chi_1,\chi_2)\,\Sigma_q^\phi(N;\chi_1,\chi_2),
\end{align}
where
\[
A_q(N;\chi_1,\chi_2)
:=\frac{\tau(\overline{\chi_1\chi_0^{(q)}})\tau(\overline{\chi_2\chi_0^{(q)}})}
{\varphi(q)^2}\,c_{\chi_1\chi_2\chi_0^{(q)}}(N),
\]
\begin{align}\label{eq:Sigma-q}
\Sigma_q^\phi(N;\chi_1,\chi_2)
:=\sum_{n_1,n_2\ge 1}\Lambda(n_1)\Lambda(n_2)\,
\phi(n_1/N)\phi(n_2/N)\,\chi_1(n_1)\chi_2(n_2)\,G_{T_q}(N-n_1-n_2).
\end{align}

For each $n_1$ in the support of $\phi(\cdot/N)\subset(N/5,4N/5)$,
we notice that the function 
\\$t\mapsto\phi(n_1/N)\phi(t/N)G_{T_q}(N-n_1-t)$ is smooth
and supported inside $(N/5,4N/5)$, hence bounded away from $0$.
The following standard result therefore applies with $S=N/5$.
 
\begin{lemma}[Smoothed explicit formula]\label{lem:smooth-EF}
Let $\chi$ be a primitive character and $\Psi\in C_c^\infty(S,\infty)$
for some $S>0$.  For every $A>0$,
\begin{align}\label{eq:smooth-EF}
\sum_{n\ge 1}\Lambda(n)\chi(n)\Psi(n)
=\1_{\chi=\chi_0^{(1)}}\int_0^\infty\Psi(t)\,dt
-\sum_{\rho}\int_0^\infty\Psi(t)\,t^{\rho-1}\,dt
+O_A(S^{-A}),
\end{align}
where the sum runs over \emph{all} zeros of $L(s,\chi)$, trivial and
non-trivial, converges absolutely, and the implied constant depends
only on $A$ and the $C^\infty$ seminorms of $\Psi$.
\end{lemma}

\subsection{The fully explicit formula}
 
We apply Lemma~\ref{lem:smooth-EF} twice to each term of
\eqref{eq:bR-character}.
For fixed $q$, $\chi_1$, $\chi_2$, set
$\Psi_{n_1}(t):=\phi(n_1/N)\phi(t/N)G_{T_q}(N-n_1-t)
\in C_c^\infty(N/5,\infty)$
for each $n_1$ in the support of $\phi(\cdot/N)$.
Applying Lemma~\ref{lem:smooth-EF} to the inner sum over $n_2$ gives
\begin{align}\label{eq:EF-inner}
\sum_{n_2\ge 1}\Lambda(n_2)\chi_2(n_2)\phi(n_2/N)G_{T_q}(N-n_1-n_2)
&=\1_{\chi_2=\chi_0^{(1)}}\int_0^\infty\Psi_{n_1}(t)\,dt \notag\\
&\quad-\sum_{\rho_2}\int_0^\infty\Psi_{n_1}(t)\,t^{\rho_2-1}\,dt
+O_A(N^{-A}),
\end{align}
where the zero sum runs over all zeros of $L(s,\chi_2)$.
Substituting into \eqref{eq:Sigma-q} and applying
Lemma~\ref{lem:smooth-EF} to the outer sum over $n_1$,
with smooth weight
$u\mapsto\phi(u/N)\int_0^\infty\phi(t/N)G_{T_q}(N-u-t)\,t^{\rho_2-1}\,dt
\in C_c^\infty(N/5,\infty)$, gives
\begin{align}\label{eq:Sigma-expanded}
\Sigma_q^\phi(N;\chi_1,\chi_2)
&=\1_{\chi_1=\chi_0^{(1)}}\1_{\chi_2=\chi_0^{(1)}}
   I_q^\phi(N;1,1) \tag*{(pole--pole)} \notag\\
&\quad-\1_{\chi_2=\chi_0^{(1)}}\sum_{\rho_1}I_q^\phi(N;\rho_1,1)
   -\1_{\chi_1=\chi_0^{(1)}}\sum_{\rho_2}I_q^\phi(N;1,\rho_2)
   \tag*{(pole--zero)} \notag\\
&\quad+\sum_{\rho_1,\rho_2}I_q^\phi(N;\rho_1,\rho_2)
   +O_A(N^{-A}), \tag*{(zero--zero)}
\end{align}
where all zero sums run over all zeros (trivial and non-trivial)
of the respective $L$-functions, and the archimedean factor is
\begin{align}\label{eq:Iq-def}
I_q^\phi(N;\rho_1,\rho_2)
:=\iint_{u_1,u_2>0}\phi(u_1/N)\phi(u_2/N)\,
u_1^{\rho_1-1}u_2^{\rho_2-1}\,G_{T_q}(N-u_1-u_2)\,du_1\,du_2.
\end{align}
 
Let $\mathcal Z_R$ denote the set of pairs $(\rho,\chi)$ with $\chi$
primitive of conductor $r\le R$ and $\rho$ any zero of $L(s,\chi)$.
Define
\begin{align}
\label{eq:Mfef}
\mathcal M(N;R)
&:=-\sum_{(\rho,\chi)\in\mathcal Z_R}
\sum_{\substack{q\le R\\\mathrm{cond}(\chi)\mid q}}
A_q(N;\chi_0^{(1)},\chi)\,I_q^\phi(N;1,\rho) \notag\\
&\phantom{:=}
-\sum_{(\rho,\chi)\in\mathcal Z_R}
\sum_{\substack{q\le R\\\mathrm{cond}(\chi)\mid q}}
A_q(N;\chi,\chi_0^{(1)})\,I_q^\phi(N;\rho,1),\\[1ex]
\label{eq:Zfef}
\mathcal Z(N;R)
&:=\sum_{\substack{(\rho_1,\chi_1),(\rho_2,\chi_2)\in\mathcal Z_R}}
\sum_{\substack{q\le R\\{[\mathrm{cond}(\chi_1),\mathrm{cond}(\chi_2)]}\mid q}}
A_q(N;\chi_1,\chi_2)\,I_q^\phi(N;\rho_1,\rho_2).
\end{align}

In conclusion we obtain the following proposition.
\begin{proposition}\label{prop:fef}
Let $R=X^\vartheta$ with $0<\vartheta<4/9$, $\eta=R^{-1/4}$, and
$T_q=\eta Xq/R$.  Then
\begin{align}\label{eq:fef-meansq}
\sum_{N\le X}
\bigl|r_\phi(N)-N\mathfrak S(N)-\mathcal M(N;R)-\mathcal Z(N;R)\bigr|^2
\ll\bigl(X^{3-\vartheta}+X^{13/5}\bigr)(\log X)^5.
\end{align}
\end{proposition}
 
\subsection{The archimedean factor and the connection with Pintz}
 We now analyse the behaviour of $I_q^\phi(N;\rho_1,\rho_2)$, depending on the location of $\rho_1,\rho_2.$
\begin{lemma}\label{lem:Iq-eval}
For $0<\Re\rho_i\le 1$, define the weighted beta integral
\begin{align}\label{eq:Bphi-def}
B_\phi(\rho_1,\rho_2):=\int_0^1\phi(\sigma)\phi(1-\sigma)\,\sigma^{\rho_1-1}(1-\sigma)^{\rho_2-1}\,d\sigma,
\end{align}
which satisfies $B_\phi(1,1)=c_\phi$.
Then for every $A>0$,
\begin{align}\label{eq:Iq-decay}
\bigl|I_q^\phi(N;\rho_1,\rho_2)\bigr|
\ll_A
N^{\Re(\rho_1+\rho_2)-1}
\Bigl(1+\bigl(|\Im\rho_1|+|\Im\rho_2|\bigr)\tfrac{T_q}{N}\Bigr)^{-A},
\end{align}
and in the region $|\Im(\rho_1+\rho_2)|\,T_q\le N$,
\begin{align}\label{eq:Iq-near}
I_q^\phi(N;\rho_1,\rho_2)
=
B_\phi(\rho_1,\rho_2)\,N^{\rho_1+\rho_2-1}
+O\!\left(
N^{\Re(\rho_1+\rho_2)-1}
\bigl(1+|\Im(\rho_1+\rho_2)|\bigr)\frac{T_q}{N}
\right).
\end{align}
\end{lemma}

\begin{proof}
Substitute $v=u_1+u_2$, $u_1=v\sigma$ in \eqref{eq:Iq-def}:
\begin{align}\label{eq:Iq-inner}
I_q^\phi(N;\rho_1,\rho_2)
=\int_0^\infty v^{\rho_1+\rho_2-1}G_{T_q}(N-v)\,J_v(\rho_1,\rho_2)\,dv,
\end{align}
where $J_v(\rho_1,\rho_2):=\int_0^1\phi(v\sigma/N)\phi(v(1-\sigma)/N)
\sigma^{\rho_1-1}(1-\sigma)^{\rho_2-1}\,d\sigma$.
On the support of $G_{T_q}(N-v)$ we have $v\asymp N$, and the cutoff
is smooth in $\sigma\in(1/5,4/5)$ with all $\sigma$-derivatives $O(1)$.
Write $t_j=\Im\rho_j$.

For \eqref{eq:Iq-decay}, split on the sign of $t_1 t_2$. If $t_1 t_2<0$,
the phase $t_1\log\sigma+t_2\log(1-\sigma)$ has derivative
$t_1/\sigma-t_2/(1-\sigma)$ of size $\ge |t_1|+|t_2|$ with no
stationary point on $(1/5,4/5)$, so repeated integration by parts in
$\sigma$ gives $|J_v|\ll_A(1+|t_1|+|t_2|)^{-A}$; the outer integral is
trivially $O(1)$, which is stronger than \eqref{eq:Iq-decay} since
$T_q\le N$. If $t_1 t_2\ge 0$, then $|t_1+t_2|=|t_1|+|t_2|$ and we extract
decay from $v$ instead: write $v^{\rho_1+\rho_2-1}=
v^{\Re(\rho_1+\rho_2)-1}e^{i(t_1+t_2)\log v}$. The phase
$(t_1+t_2)\log v$ has derivative $\asymp(|t_1|+|t_2|)/N$ on $v\asymp N$,
and the amplitude
$v^{\Re(\rho_1+\rho_2)-1}G_{T_q}(N-v)J_v$ has $k$-th $v$-derivative
$O(N^{\Re(\rho_1+\rho_2)-1}T_q^{-k})$, the scale being set by $G_{T_q}$.
Integration by parts $A$ times yields \eqref{eq:Iq-decay}.

For \eqref{eq:Iq-near}: substitute $w=(N-v)/T_q$ to write
\[
I_q^\phi
=\int_{-2}^{2}G(w)\,(N-T_qw)^{\rho_1+\rho_2-1}\,
J_{N-T_qw}(\rho_1,\rho_2)\,dw.
\]
Since $|T_qw/N|\le 2T_q/N\ll 1$ on the support of $G$,
Taylor expansion of $\phi$ gives\\
$J_{N-T_qw}(\rho_1,\rho_2)=B_\phi(\rho_1,\rho_2)+O(T_q/N)$.
For the factor $(N-T_qw)^{s-1}$ with $s=\rho_1+\rho_2$,
writing $(N-T_qw)^{s-1}=N^{s-1}(1-T_qw/N)^{s-1}$
and expanding $|e^{(s-1)\log(1-T_qw/N)}-1|\ll|s|T_q/N\ll(1+|\Im s|)T_q/N$
(using $|s-1|\ll 1+|\Im s|$ for $\Re s\in(0,2]$),
then integrating against $G(w)$ and using $\int_{-2}^2 G(w)\,dw=1$
gives \eqref{eq:Iq-near}.
\end{proof}
 
The lemma identifies the connection with Pintz's formula \eqref{eq:pintz-EF}.
For any fixed $R$, the set of zero pairs $(\rho_1,\chi_1),(\rho_2,\chi_2)$
with $|\Im\rho_i|\le T_0$ for a suitable threshold $T_0=T_0(R)$ is finite. The leading term in \eqref{eq:Iq-near} becomes independent of $q$ and we can generate the generalised exceptional series of \eqref{eq:genSS-def}. While the saving is not very strong, because he only considers a finite amount of pairs, it is sufficient. In general, this lemma shows that our intervals of length $T_q$ make zeros with $|\Im\rho_i|$ somewhat larger than $T_q/N$ negligible via the decay in \eqref{eq:Iq-decay}. It should be noted, though, that the inclusion of an $\eta=R^{-1/4}$ term, makes our intervals shorter than usual. It would be interesting to see if other ideas could circumvent this interval shortening, perhaps ones related to Heath-Brown's circle method \cite{hbcircle}.

\section{A sparse Hardy--Littlewood conjecture and Siegel  Zeros}
\label{sec:sparse-HL}

We close with a simple consequence of the exceptional-zero case of Pintz's
work: knowing the expected number of Goldbach representations, even in a
rather sparse form, rules out the existence of exceptional zeros . Results
of this type have attracted considerable recent attention: Fei \cite{F2016},
the first author and Halupczok \cite{BH2021}, Jia \cite{Jia2022} and Goldston--Suriajaya
\cite{GS2021} (see also Friedlander--Goldston--Iwaniec--Suriajaya
\cite{FGIS2022}) showed that if $r_2(N)$ stays close to the Hardy--Littlewood
prediction $\mathfrak{S}(N)N$ of Conjecture~\ref{conj:HL} for essentially all
even $N$, then no exceptional zero can exist. Our observation is that the
machinery behind the power-saving exceptional-set bounds, specifically the
exceptional-zero case of Pintz's explicit formula, upgrades such
statements without much effort: the asymptotic may in addition fail on a
power-sized exceptional set (Theorem~\ref{thm:sparse-HL-no-Siegel}). This
approach does fall short of the more recent result of Matom\"aki--Merikoski
\cite[Cor.~1.2]{MM2023}, who need only a \emph{single} multiple $N$ of the
conductor with roughly the expected number of representations. However, since their
hypothesis lives at the scale $N\ge\widetilde r^{\,10}$, while ours operates
at $N\asymp\widetilde r^{\,A}$ for any fixed $A>5/2$, their statement
does not imply ours.

Throughout this section, $\widetilde\chi$ denotes a primitive real character
of conductor $\widetilde r$ whose $L$-function has a real zero
$\widetilde\beta=1-\widetilde{\delta}$ near $1$. As in \eqref{eq:MVsiegelsecondary} and
\eqref{eq:beta-bound} we use the beta factor
\begin{align*}
    B(\varrho_1,\varrho_2):=\frac{\Gamma(\varrho_1)\Gamma(\varrho_2)}{\Gamma(\varrho_1+\varrho_2)},
\end{align*}
so that $B(\widetilde\beta,\widetilde\beta)\to1$ as $\widetilde\beta\to1$.
Beyond the general bound \eqref{eq:genSS-bound}, we need two facts about the
generalised singular series \eqref{eq:genSS-def} attached to
$\widetilde\chi$, both direct consequences of the closed-form evaluation in
the Main Lemma of \cite{P2023} ((7.4)--(7.7) there): for even $N$,
\begin{equation}\label{eq:ss-bounds}
    \bigl|\mathfrak{S}(\chi_0^{(1)},\widetilde\chi,N)\bigr|
        \ll \mathfrak{S}(N)\,\frac{\widetilde r}{\varphi(\widetilde r)^{2}},
    \qquad
    \mathfrak{S}(\widetilde\chi,\widetilde\chi,N)
        =\widetilde\chi(-1)\,\mathfrak{S}(N)
    \quad\text{if }\widetilde r\mid N.
\end{equation}

Fix $2/5\le\vartheta<4/9$, set $R=X^{\vartheta}$ and write
$\mathcal L:=\log X$. The key structural input is the Deuring--Heilbronn
phenomenon, which we already encountered in Proposition~\ref{prop:gallagher}: a
Siegel zero pushes all other zeros  away from the line $\Re s=1$. In the
quantitative form of \cite[Lem.~4.22]{P2023}, the zero $\widetilde\beta$
forces every other zero $\varrho$ of every $L(s,\chi)$ with
$\mathrm{cond}(\chi)\le R$ and $|\Im\varrho|\le\sqrt X$ to satisfy
\begin{align*}
    1-\Re\varrho\ \ge\ c_0(\vartheta)\,
    \frac{\log\bigl(1/(\widetilde{\delta}\mathcal L)\bigr)}{\mathcal L}.
\end{align*}
Consequently, for every cutoff $H$ with
\begin{equation}\label{eq:H-range}
    0<H\le H_1:=c_0(\vartheta)\log\frac{1}{\widetilde{\delta}\mathcal L},
\end{equation}
the set $\mathcal{E}=\mathcal{E}(H,\sqrt X,R,X)$ of generalised exceptional
zeros  defined in \eqref{eq:EHT-def} contains no zero besides
$(\widetilde\beta,\widetilde\chi)$.

\begin{proposition}[Exceptional-zero Goldbach formula]\label{prop:EZ-Goldbach}
Let $\vartheta$, $R$ be as above and suppose that $\widetilde r\le R$ and
$0<\widetilde{\delta}<h/\mathcal L$, where $h=h(\vartheta)>0$ is sufficiently small
(cf.\ \cite[Thm~2]{P2023}). There are $c'=c'(\vartheta)>0$ and
$\varepsilon_0=\varepsilon_0(\vartheta)>0$ such that for every
$\varepsilon>0$ and all but $O_\varepsilon(X^{3/5+\varepsilon})$ even
$N\in[X/2,X]$ we have
\begin{align}\label{eq:EZ-Goldbach}
    r_2(N)
    =&\mathfrak{S}(N)\,N
    +\mathfrak{S}(\widetilde\chi,\widetilde\chi,N)\,
        B(\widetilde\beta,\widetilde\beta)\,N^{2\widetilde\beta-1}
    \\&+O\!\bigl(\mathfrak{S}(N)\,X\,(\widetilde{\delta}\mathcal L)^{c'}\bigr)
    +O\!\Bigl(\mathfrak{S}(N)\,N\,\frac{\widetilde r}{\varphi(\widetilde r)^{2}}\Bigr)
    +O\!\bigl(X^{1-\varepsilon_0}\bigr), \nonumber
\end{align}
where the implied constants depend only on $\vartheta$.
\end{proposition}

\begin{proof}[Sketch]
The starting point is Pintz's explicit formula for the major arcs
\cite[Thm~1]{P2023}, which refines the version recorded in
Theorem~\ref{thm:pintz-explicit} in two ways that matter here. First, the
pole $\varrho_0=1$ is included with the zeros, with signs $A(\varrho_0)=+1$
and $A(\varrho)=-1$ for genuine  zeros \cite[(2.8)]{P2023}, so that the
pole--zero pairs, which were dropped as mixed terms below \eqref{eq:EM-def},
are kept. Secondly, the density-estimate error carries the factor
$\mathfrak{S}(N)$ (it is assembled in \cite[\S8--\S9, cf.\ (9.16)]{P2023}).
Up to the contribution of prime powers and of terms with
$n_i\le X^{1-\varepsilon_0}$, which is $O(X^{1-\varepsilon_0/2})$, the
formula reads, for any cutoff $H$,
\begin{align*}
    r_{\mathfrak M}(N)
    =\sum_{(\varrho_i,\chi_i),\,(\varrho_j,\chi_j)}
    A(\varrho_i)A(\varrho_j)\,\mathfrak{S}(\chi_i,\chi_j,N)\,
    I(\varrho_i,\varrho_j,N)
    +O\!\bigl(\mathfrak{S}(N)Xe^{-cH}\bigr)+O\!\bigl(X^{1-\varepsilon_0}\bigr),
\end{align*}
where $c=c(\vartheta)>0$, the sum runs over the pairs from
$\mathcal{E}\cup\{(\varrho_0,\chi_0^{(1)})\}$ and, by
\cite[Lem.~4.9]{P2023},
\begin{align*}
    I(\varrho_1,\varrho_2,N)
    :=\sum_{\substack{k+\ell=N\\ k,\ell>X^{1-\varepsilon_0}}}
    k^{\varrho_1-1}\ell^{\varrho_2-1}
    =B(\varrho_1,\varrho_2)\,N^{\varrho_1+\varrho_2-1}
    +O\bigl(X^{1-\varepsilon_0}\bigr).
\end{align*}
By \eqref{eq:H-range} the sum runs over the four pairs formed from
$\varrho_0$ and $\varrho_1:=\widetilde\beta$ only. The pair
$(\varrho_0,\varrho_0)$ gives the main term
$\mathfrak{S}(N)N+O(\mathfrak{S}(N)X^{1-\varepsilon_0})$, the two mixed
pairs contribute $O(\mathfrak{S}(N)N\widetilde r/\varphi(\widetilde r)^{2})$
by the first bound in \eqref{eq:ss-bounds}, and $(\varrho_1,\varrho_1)$
gives the secondary main term in \eqref{eq:EZ-Goldbach}. Choosing the
largest admissible cutoff $H=H_1$ in \eqref{eq:H-range} turns the error
$O(\mathfrak{S}(N)Xe^{-cH})$ into
$O(\mathfrak{S}(N)X(\widetilde{\delta}\mathcal L)^{c'})$ with $c'$ depending on $\vartheta$.

It remains to add the minor arcs,
$r_2(N)=r_{\mathfrak M}(N)+r_{\mathfrak m}(N)$ by
\eqref{eq:splitmajorminor}. By Lemma~\ref{lem:minor-arc-meansq}, or
\cite[(5.3)]{P2023}, we have $|r_{\mathfrak m}(N)|\le X^{1-\varepsilon_0}$
(after decreasing $\varepsilon_0$ if necessary) for all even $N\le X$ with
at most $O_\varepsilon(X^{1+\varepsilon}/R+X^{3/5+\varepsilon})$ exceptions,
and the first term is dominated by the second precisely when
$\vartheta\ge2/5$. The exponent $3/5$, and with it the restriction
$\vartheta\ge 2/5$, is forced by Vinogradov's minor-arc bound
(Proposition~\ref{prop:vino}).
\end{proof}

We are now ready to state and prove the proposed new result connecting Goldbach representations with the (non-)existence of Siegel--zeros.

\begin{theorem}\label{thm:sparse-HL-no-Siegel}
Fix $A>5/2$ and $\delta\in(0,1)$. There are $c=c(\delta)>0$ and
$r_0=r_0(A,\delta)$ such that the following holds for every
$\widetilde r\ge r_0$. Let $X=\widetilde{r}^A$ and assume that
\begin{equation}\label{eq:sparse-HL}
    \delta\,\mathfrak{S}(N)N\ \le\ r_2(N)\ \le\ (2-\delta)\,\mathfrak{S}(N)N
\end{equation}
holds for all but at most $X^{3/5}$ even $N\in[X/2,X]$ with
$\widetilde r\mid N$, then no $L(s,\widetilde\chi)$ with $\widetilde\chi$ a
primitive real character modulo $\widetilde r$ has a real zero
$\widetilde\beta>1-c/\log X$.
\end{theorem}

\begin{proof}
Take $\vartheta=2/5$, so that $\widetilde r=X^{1/A}<X^{2/5}=R$ because
$A>5/2$, and suppose that some primitive real character
$\widetilde\chi\pmod{\widetilde r}$ has a real zero
$\widetilde\beta=1-\delta_1$ with $\delta_1\mathcal L\le c$. For $c<h$
Proposition~\ref{prop:EZ-Goldbach} applies, and by $X\le 2N$ its three
error terms are in total
\begin{align*}
    \ll\ \mathfrak{S}(N)\,N\Bigl(c^{\,c'}
    +\frac{\widetilde r}{\varphi(\widetilde r)^{2}}\Bigr)+X^{1-\varepsilon_0},
\end{align*}
which is at most $\tfrac{\delta}{4}\,\mathfrak{S}(N)N$ for all even
$N\in[X/2,X]$, once $c=c(\delta)$ is small enough and
$\widetilde r\ge r_0(A,\delta)$. Moreover,
$B(\widetilde\beta,\widetilde\beta)N^{-2\delta_1}=1+O(c)$ on this range.

Restrict now to the even multiples of $\widetilde r$ in $[X/2,X]$, of which
there are $\asymp X^{1-1/A}$. Since $1-1/A>3/5$, applying
Proposition~\ref{prop:EZ-Goldbach} with
$\varepsilon=\tfrac12(1-1/A-3/5)$ and discarding also the at most $X^{3/5}$
integers excluded in \eqref{eq:sparse-HL} removes only
$O_A(X^{3/5+\varepsilon})=o(X^{1-1/A})$ of them, so at least one multiple
$N$ remains for $\widetilde r\ge r_0(A,\delta)$. For this $N$ we have
$\widetilde r\mid N$, so \eqref{eq:EZ-Goldbach} and the identity
$\mathfrak{S}(\widetilde\chi,\widetilde\chi,N)
=\widetilde\chi(-1)\mathfrak{S}(N)$ from \eqref{eq:ss-bounds} give
\[
    r_2(N)=\Bigl(1+\widetilde\chi(-1)\bigl(1+O(c)\bigr)\Bigr)\mathfrak{S}(N)N
    +E(N),
    \qquad |E(N)|\le\frac{\delta}{4}\,\mathfrak{S}(N)N.
\]
Once $c=c(\delta)$ is small enough, this yields
$r_2(N)<\delta\,\mathfrak{S}(N)N$ when $\widetilde\chi(-1)=-1$, violating
the lower bound in \eqref{eq:sparse-HL}, and
$r_2(N)>(2-\delta)\,\mathfrak{S}(N)N$ when $\widetilde\chi(-1)=+1$,
violating the upper bound. Hence every real zero satisfies
$\delta_1\mathcal L>c$, which is the claim.
\end{proof}

\section*{Acknowledgements}

La première auteure remercie le CDP C2EMPI pour son soutien, ainsi que l'\'Etat Francais
dans le cadre du programme France-2030,
l'Universit\'e de Lille, l'Initiative d'Excellence de l'Université de Lille, la Métropole
Européenne de Lille pour leur financement et leur appui au projet 
R-CDP-24-004-C2EMPI. This project has
received funding from the European Research Council (ERC) under the European Union's Horizon research and innovation programme (grant No. 101162746 second author). 

\bibliographystyle{plain}
\bibliography{references}
\end{document}